\documentclass[11pt,letterpaper]{article}
\usepackage[margin = 1in]{geometry}

\usepackage{lipsum}
\usepackage{amsfonts}
\usepackage{graphicx}
\graphicspath{ {images/} }
\usepackage{epstopdf}
\usepackage{dsfont}
\usepackage{mathtools}
\usepackage{empheq}
\usepackage{amsfonts}
\usepackage{amsgen,amsthm,amsmath,amstext,amsbsy,amsopn,amssymb,subcaption,stmaryrd,mathabx,mathrsfs}
\usepackage{comment}
\usepackage[font=footnotesize,labelfont=bf]{caption}

\usepackage{blkarray}
\usepackage{url}
\usepackage{longtable}
\usepackage{mathtools}
\usepackage{multirow}
\usepackage{array}

\usepackage{enumitem}
\usepackage{hyperref}
\usepackage{natbib}
\usepackage{booktabs}
\usepackage{algorithm}
\usepackage{algpseudocode}
\algrenewcommand\algorithmiccomment[1]{\hfill\texttt{//~#1}}
\usepackage{xcolor}
\usepackage[utf8]{inputenc} 
\usepackage[T1]{fontenc}

\ifpdf
  \DeclareGraphicsExtensions{.eps,.pdf,.png,.jpg}
\else
  \DeclareGraphicsExtensions{.eps}
\fi

\newcommand{\cF}{\mathcal{F}}

\newcommand{\cI}{\mathcal{I}}

\newcommand{\cL}{\mathcal{L}}

\newcommand{\cP}{\mathcal{P}}

\newcommand{\cR}{\mathcal{R}}
\newcommand{\cS}{\mathcal{S}}

\newcommand{\cW}{\mathcal{W}}
\newcommand{\cX}{\mathcal{X}}

\newcommand{\bbC}{\mathbb{C}}
\newcommand{\bbD}{\mathbb{D}}

\newcommand{\bbN}{\mathbb{N}}

\newcommand{\bbQ}{\mathbb{Q}}
\newcommand{\bbR}{\mathbb{R}}

\newcommand{\bbZ}{\mathbb{Z}}

\DeclarePairedDelimiter\ceil{\lceil}{\rceil}
\DeclarePairedDelimiter\floor{\lfloor}{\rfloor}
\DeclarePairedDelimiter\abs{\lvert}{\rvert}
\DeclarePairedDelimiter\norm{\lVert}{\rVert}

\newcommand{\dif}{\,\mathrm{d}}
\newcommand{\im}{\mathrm{i}}
\newcommand{\iid}{\overset{\mathrm{iid.}}{\sim}}
\newcommand{\define}{\overset{\mathrm{def}}{=}}

\DeclareMathOperator{\E}{\mathbb{E}} 
\newcommand{\indi}{{\mathds{1}}}
\newcommand{\Arg}{{\mathrm{Arg}}}

\DeclareMathOperator*{\argmax}{\arg\max}

\DeclareMathOperator\arctanh{arctanh}
\DeclareMathOperator\arccosh{arccosh}
\DeclareMathOperator\Log{Log}

\newtheorem{Theorem}{Theorem}

\newtheorem{Lemma}{Lemma}

\newtheorem{Proposition}{Proposition}

\newtheorem{Claim}{Claim}

\DeclareMathAlphabet\mathbfcal{OMS}{cmsy}{b}{n}

\usepackage{bbm}

\hypersetup{
    colorlinks=true,
    linkcolor=blue,
    filecolor=blue,      
    urlcolor=cyan,
    citecolor=blue
}

\makeatletter
\newcommand*{\rom}[1]{\expandafter\@slowromancap\romannumeral #1@}
\makeatother

\allowdisplaybreaks

\begin{document}
	\title{Sample Complexities of Estimating Gumbel--Max Watermark Proportions with and without Reduction to Pivotal Statistics
    }
    
	\author{
    Shuwen Chai\thanks{Equal contribution. Authors are listed in alphabetical order.}\\
    Northwestern University
    \and
    Qiaosen Wang\footnotemark[1]\\
    University of Chicago
}
	\date{\today}
	\maketitle

\begin{abstract}
    Watermarking promises statistical traceability of large language model (LLM) uses, but real documents rarely arrive as purely human-written or purely LLM-generated. This motivates a quantitative question beyond detection: what proportion of a document is generated from a pre-specified watermarked LLM? We study this watermark proportion estimation problem under the Gumbel--max watermarking mechanism, treating the next-token prediction distributions as unknown and arbitrary nuisance parameters subject to a non-degeneracy condition. We compare two observation regimes: in the full observation regime, the estimator observes the pseudorandom vector and the selected token at each position; in the more prevalent setting of pivotal reduction, it observes only a scalar pivot, which follows a one-dimensional Uniform--Beta mixture distribution. 
    Under pivotal reduction, we develop a Laguerre-polynomial estimator and establish a matching information-theoretic lower bound for the sample complexity. For full observation, we introduce an event-counting estimator and show a matching lower bound, yielding a substantially smaller sample complexity. As our results imply, although reducing to pivotal statistics is an elegant and prevalent choice, it is not always sample-efficient for estimating the proportion of watermarks.
\end{abstract}
\tableofcontents

\begin{sloppypar}

\section{Introduction}
As large language model (LLM) outputs become increasingly interleaved with human-written text, watermarking has emerged as a promising way to leave a statistical trace for the detection of potential LLM use \cite{aaronson2023watermarking, kirchenbauer2023watermark, kuditipudi2024robust, hu2024unbiased, christ2024undetectable, zhao2024provable}. In practice, however, human-AI collaboration has become highly efficient and prevalent, so that most documents are neither purely LLM-generated nor purely human-written. 
A human-written document may be appended with LLM-generated completion or vice versa, making binary Yes-or-No detection of LLM involvement not the end of the story. 
People are curious about a more quantitative question: from sampled text, can one recover the fraction of a document that carries a pre-specified watermark signal?
We call this estimand the watermark proportion and study the sharp sample complexity of estimating it under the Gumbel--max watermarking mechanism~\cite{aaronson2023watermarking}.

Denote by $\cW$ the token alphabet and by $\abs{\cW}$ its cardinality. 
At each time $t$, a probabilistic token generator, either human or LLM, performs next-token prediction (NTP) by sampling a token $w_t$ from a multinomial distribution $P_t=(p_{t,w})_{w\in\cW}$, where $p_{t,w}\geq 0$ and $\sum_{w\in\cW}p_{t,w}=1$. In the framework of Gumbel--max watermarking~\cite{aaronson2023watermarking}, a pseudorandom vector $U_t=(U_{t,1},U_{t,2},\cdots,U_{t,|\cW|})$ with i.i.d. $\mathrm{Unif}(0,1)$ entries is generated independently at each time\footnote{We use the standard idealized pseudorandomness assumption for the hash function.} $t=1,2,\cdots,N$, from a pseudorandom function computed based on the current position, preceding tokens, document metadata, etc. For non-watermarked contents, the token is sampled from the NTP distribution and is independent of the pseudorandom vector, that is,
\begin{align}\label{eq:nonwatermark_rule}
    (\textnormal{non-watermarked})\qquad U_t\sim \mathrm{Unif}(0,1)^{\otimes \abs{\cW}},\quad w_t\sim P_t,\quad w_t\perp U_t.
\end{align}
When the Gumbel--max watermark is activated,  the token is chosen by the following decoding rule:
\begin{align}
    \label{eq:gumbel_rule}
    w_t =  \cS(U_t,P_t)\define \argmax_{w \in \cW} \frac{\log U_{t,w}}{p_{t,w}}.
\end{align}
The Gumbel--max mechanism ensures that the marginal distribution of $w_t$ remains $P_t$ \cite{gumbel1954statistical, maddison2014sampling, jang2017categorical}. However, a token $w\in\cW$ with a larger pseudorandom value $U_{t,w}$ is more likely to be selected as the next token $w_t$. In this way, the participation of a watermarked LLM can be traced from the statistical correlation between $U_t$ and $w_t$. In contrast, for human-written natural languages or non-watermarked LLM-generated contents, the tokens are sampled from NTP distributions, totally independent of the pseudorandom vectors. Therefore, even if the LLM perfectly mimics the NTP distribution of human users, it can still be detected by inspecting the joint distribution between tokens and pseudorandom vectors.

To model hybrid documents containing both watermarked and non-watermarked tokens, we consider the following token-level mixture model, following~\cite{li2025optimal}. 
Concretely, for $t=1,2,\cdots,N$, we observe independently
\begin{align}\label{eq:full_data}
    (\textnormal{watermarked})\qquad U_t\sim \mathrm{Unif}(0,1)^{\otimes \abs{\cW}},\quad w_t\mid U_t\sim (1-\varepsilon) P_t+\varepsilon \delta_{\cS(U_t,P_t)}, 
\end{align}
where $\varepsilon\in[0,1]$ represents the proportion of watermarked content, and $\cS(U_t,P_t)$ is the aforementioned Gumbel--max decoding rule in \eqref{eq:gumbel_rule}. We treat NTP distributions $\{P_t\}_{t\in[N]}$ as unknown and arbitrary nuisance parameters. This black-box formulation accounts for the fact that the NTP distributions are typically inaccessible to external verifiers, due to commercial encapsulation by the service provider. 
The arbitrariness of $\{P_t\}_{t\in[N]}$ encompasses diverse contexts, sequential dependence structures, and temporal shifts in the token-generating process. Crucially, a regularity assumption must be adopted: we assume
\begin{align}\label{eq:not_one-hot}
    \max_{w\in\cW}\, p_{t,w}\leq 1-\Delta,\quad t=1,2,\cdots,N,
\end{align}
for some $\Delta\in(0,1)$. This assumption prevents the NTP distributions from being one-hot. A one-hot NTP distribution $P_t$ will break the correlation between $U_t$ and $w_t$, making the proportion-estimation problem statistically impossible.
In reality, one-hot $P_t$ does not exist since the intermediate values in LLMs are bounded for typical inputs and will pass through a softmax layer with a prescribed temperature.

In this paper, we study the sample complexity for estimating the Gumbel--max watermark proportion $\varepsilon$ under two different settings on the verifier side:
\begin{enumerate}
    \item \textbf{Full observation.} 
    We assume both pseudorandom vectors $\{U_t\}_{t\in[N]}$ and tokens $\{w_t\}_{t\in[N]}$ are observed, and hence the token alphabet $\cW$ is regarded as known.
    \item \textbf{Reduction to pivotal statistics.} We observe only \emph{pivotal statistics}, namely $\{Y_t\define U_{t,w_t}\}_{t\in[N]}$, which are real numbers in $(0,1)$. Since we do not observe the entire pseudorandom vector $U_t\in (0,1)^{\abs{\cW}}$ and the token $w_t\in\cW$, the vocabulary size $\abs{\cW}$ is unknown and can be an arbitrary integer no smaller than $2$.
\end{enumerate}
The pivotal reduction dates back to the germinal work of \cite{aaronson2023watermarking}, where the Gumbel--max watermark was first proposed. In \cite{aaronson2023watermarking}, a test fully based on pivotal statistics is developed to distinguish purely watermarked texts from their non-watermarked counterparts. In a substantial line of subsequent studies, the reduction to pivotal statistics is always adopted both to test the existence of Gumbel--max watermarks \cite{li2025statistical, cai2025optimal, li2026robust, he2026empirical, lattimore2026refined} and to estimate the watermark proportion \cite{li2025optimal}. The popularity of pivotal statistics is ascribed to its efficiency and effectiveness. Each pivotal statistic $Y_t=U_{t,w_t}$ serves as a tractable one-dimensional reduction of the original high-dimensional vector-token pair $(U_t,w_t)\in (0,1)^{\abs{\cW}}\times \cW$ while preserving part of the watermarking signals. In particular, the law of $Y_t$ is simply $\mathrm{Unif}(0,1)$ when the watermark is absent and becomes a mixture of Beta distributions when the corresponding token is purely from a Gumbel--max watermarked source. In the hybrid setup \eqref{eq:full_data}, the pivotal statistic follows
\begin{align}
    \label{eq:uniform_beta_mixture}
    Y_t\sim (1-\varepsilon)\mathrm{Unif}(0,1)+\varepsilon \sum_{w\in\cW}p_{t,w}\,\mathrm{Beta}(1/p_{t,w},1),
\end{align}
whose cumulative distribution function (CDF) is given by \(F(r)=(1-\varepsilon)r+\varepsilon\sum_{w\in\cW} r^{1/p_w}\).
When $\max_{w \in \cW} p_{t,w}<1$, the CDF of the pivotal statistic is distinct from the CDF of a uniform random variable. 
Therefore, $\varepsilon$ is identifiable from pivotal statistics whenever the regularity condition~\eqref{eq:not_one-hot} holds. Nevertheless, such a radical reduction may lose a non-negligible fraction of information and incur a worse minimax sample complexity than working with the full observation.

\subsection{Main Results}
In this paper, we establish the minimax-optimal sample complexities of estimating the Gumbel--max watermark proportion $\varepsilon$ under pivotal reduction and full observation. For a target accuracy \(\tau\in(0,1)\) and failure probability \(\delta\in(0,1)\), the sample-complexity question is: 

\emph{How large should the sample size \(N\) be so that the watermark proportion \(\varepsilon\) can be estimated within absolute error \(\tau\), with probability at least $1-\delta$ uniformly over all sequences of unknown NTP distributions that satisfy the regularity condition~\eqref{eq:not_one-hot}?}

\paragraph{Pivotal reduction leads to a nonparametric rate.}As we reveal, when only pivotal statistics $\{Y_t\}_{t\in[N]}$ are observed, estimating $\varepsilon$  requires a sample size of
\begin{align}\label{eq:sample_complexity_pivotal}
    N\gtrsim_\delta\, \left(1/\tau\right)^{2\vee \Theta(1/\sqrt{\Delta})}.
\end{align}
This sample complexity is equivalent to a nonparametric estimation rate of $\Omega_\delta(N^{-\{(1/2)\wedge\Theta(\sqrt{\Delta})\}})$, which, for small $\Delta$, is substantially slower than the parametric rate $N^{-1/2}$ in \cite{li2025optimal} conditioned on side information. In particular, \cite{li2025optimal} assumes additional access to purely watermarked samples $Y'_{t}\sim \sum_{w\in \cW}p_{t,w}\mathrm{Beta}(1/p_{t,w},1)$, which enables the estimation of the average CDF corresponding to the purely watermarked components. Such auxiliary samples are highly informative but rarely available in practice. The sample complexity they derive can therefore be much more optimistic than what we consider in this paper.

We also show that the lower bound \eqref{eq:sample_complexity_pivotal} can be achieved by an estimator that approximates $\varepsilon$ using Laguerre polynomials in $\{-\log(Y_t)\}_{t\in[N]}$. Our analysis breaks down the bias--variance trade-off of this estimator to a polynomial optimization problem on the complex plane:
\begin{align*}
    \inf_{f} \underbrace{\sup_{z\in[\Delta,1)}\,\abs{f(z)}}_{\textnormal{bias proxy}}\qquad \textnormal{s.t.}\ f(0)=1,\  \underbrace{\sup_{z\in\overline{\bbD}}\, \abs{f(z)}}_{\textnormal{variance proxy}}\leq A,
\end{align*}
where $A\geq 1$ and the minimization is taken over all polynomials with real coefficients that satisfy the constraints. The delicate infimum $\Theta(1)\cdot A^{-\Theta(\sqrt{\Delta})/(1/2-\Theta(\sqrt{\Delta}))}$ of the above optimization leads to the nonparametric rate.
\paragraph{Full observation restores parametric dependence on $\tau$.}When the full samples of pseudorandom vectors are provided, we show a sample complexity lower bound of
\begin{align}\label{eq:sample_complexity_full}
    N\gtrsim_{\delta,\abs{\cW}}\,(1/\tau)^2\cdot (1/\Delta)^{\abs{\cW}-1},
\end{align}
which can be achieved by an estimator based on the empirical frequency of a carefully designed event. 
Thus, in the large-sample regime where the alphabet size $\abs{\cW}$ is fixed while the sample size $n$ increases, the sample complexity of estimation under full observation~\eqref{eq:sample_complexity_full} is better than its pivot-based counterpart \eqref{eq:sample_complexity_pivotal}. Therefore, in this regime, the full samples contain strictly more information than the pivotal statistics, revealing that pivotal reduction is not always a loss-free choice.

\subsection{Related Work}

\paragraph{Gumbel--max watermark and reduction to pivotal statistics.}
The Gumbel--max mechanism was first introduced to statistical watermarking of LLMs by Aaronson and Kirchner in \cite{aaronson2023watermarking}, while the mechanism itself dates back to much earlier lines of work including~\cite{gumbel1954statistical, maddison2014sampling, jang2017categorical}. The reduction to pivotal statistics $\{Y_t=U_{t,w_t}\}_{t\in[N]}$ was first studied by \cite{aaronson2023watermarking}, where the authors use \( -\sum_{t=1}^N\log(1-Y_t)\) to detect:
\begin{align*}
    H_0:\ \varepsilon=0\ \textnormal{in \eqref{eq:full_data}}\qquad\textnormal{versus}\qquad H_1:\ \varepsilon=1\ \textnormal{in \eqref{eq:full_data}}.
\end{align*}
Following \cite{aaronson2023watermarking}, a line of subsequent work \cite{li2025statistical, cai2025optimal, he2026empirical, lattimore2026refined} also explores the same detection problem. In particular, the work~\cite{li2025statistical} considers all tests in the sum-score form $\indi\{\sum_{t=1}^N h(Y_t)\geq \alpha\}$ for a threshold $\alpha$ and a score function $h$. 
They find the score function that is optimal under their notion of efficiency rate, and show that the score $h(y)=-\log(1-y)$ chosen by \cite{aaronson2023watermarking} is sub-optimal. 
Another issue is raised in~\cite{cai2025optimal}: repeated context windows may undermine an imperfect pseudorandom generator and incur dependence between the pivotal statistics. They propose a test based on modified score functions that is optimal under the same notion of optimality as in \cite{li2025statistical}. 
Empirical comparisons between pivot-based testing procedures in~\cite{he2026empirical} suggest that sum-score tests previously studied in \cite{aaronson2023watermarking,li2025statistical,cai2025optimal} might not be optimal over all pivot-based tests. Lattimore~\cite{lattimore2026refined} replaces the score function $h(y)=-\log(1-y)$ in \cite{aaronson2023watermarking} with a truncated power-law of $1-Y_t$ and proves its Neyman--Pearson-type near-optimality over pivot-based tests.

The work \cite{li2026robust} initiates the study of the hybrid setup \eqref{eq:full_data} in which the tokens are sampled from a mixture of watermarked and non-watermarked sources. The focus therein remains on pivot-based detection, but the alternative hypothesis is changed from the purely watermarked version to $H_1: \varepsilon\asymp n^{-p}$ for $p\in[0,1]$. 
Subsequently, \cite{li2025optimal} moves from the detection of a nonzero $\varepsilon$ to estimating how large the proportion $\varepsilon$ is, which is most relevant to the goal of this paper. Specifically, the approach therein leverages the average cumulative distribution function (CDF) of the pivotal statistics:
\begin{align*}
    \overline{F}(r)=(1-\varepsilon)F_0(r)+\varepsilon\overline{F}_{P_{1:N}}(r),\quad r\in(0,1),
\end{align*}
where $F_0(r)\equiv r$ is the CDF of $\mathrm{Unif}(0,1)$, $F_{P_t}(r)=\sum_{w\in\cW} p_{t,w}r^{1/p_{t,w}}$, and $\overline{F}_{P_{1:N}}(r)=\sum_{t=1}^N F_{P_t}(r)/N$. It is observed that
\begin{align*}
    \varepsilon=\frac{\E_{F_0}[v]-\E_{\overline{F}}[v]}{\E_{F_0}[v]-\E_{\overline{F}_{P_{1:N}}}[v]}
\end{align*}
for any weight function $v:(0,1)\mapsto \bbR$ such that all the above expectations exist. Given a function $v$, $\E_{F_0}[v]$ is a known value and $\E_{\overline{F}}[v]$ can be estimated by the empirical average $\sum_{t=1}^N v(Y_t)/N$. However, $\E_{\overline{F}_{P_{1:N}}}[v]$, the expectation under the averaged pure watermark distribution, is not accessible from the pivotal samples $Y_{1:N}$. To overcome this issue, \cite{li2025optimal} assumes access to auxiliary samples $\{Y'_t\sim F_{P'_t}\}_{t\in N'}$ that are purely watermarked under NTP distributions $\{P'_t\}_{t\in N'}$ whose induced average pivotal CDF is close to that of $\{P_t\}_{t\in[N]}$. With this side information, $\E_{\overline{F}_{P_{1:N}}}[v]$ can be approximated with the auxiliary samples by $\sum_{t=1}^{N'}v(Y'_t)/N'$. This approach yields a rate of
\begin{align*}
    \abs{\widehat{\varepsilon}-\varepsilon}\lesssim \frac{\alpha_N}{\sqrt{N}}+\beta_N,
\end{align*}
where $\alpha_N$ depends on $\varepsilon$ and the density ratio $\dif \overline{F}_{P_{1:N}}/\dif F_0$, the magnitude of $\beta_N$ depends on how well this density ratio can be estimated using the auxiliary data. In contrast, our paper studies pivot-based proportion estimation in a more fundamental setting in which purely watermarked auxiliary samples are absent. As our theorems imply, the parametric rate $N^{-1/2}$ is unattainable when only pivotal samples from hybrid sources are available, and any estimator must suffer from a much slower nonparametric rate of $N^{-\{(1/2)\wedge c\sqrt{\Delta}\}}$.

Surprisingly, the aforementioned works on either the detection or the proportion estimation of Gumbel--max watermarks are all pivot-based: their procedures use only pivotal statistics, and the optimality conditions therein are established for pivot-based tests/estimators. The only known result that uses the full sample $(U_t,w_t)$ is the testing procedure in \cite{li2025likelihood}, which assumes that the NTP distributions are known or consistently estimable. Therefore, it remains open whether, and to what extent, the Gumbel--max watermark proportion can be estimated in an NTP-agnostic manner, which is among the motivations of this paper.

\paragraph{Polynomial approximation in statistical learning.} In the pivot-based setting, we estimate $\varepsilon$ by approximating it with Laguerre polynomials of transformed samples, and our lower bound construction is also guided by approximation-theoretic facts of polynomials. Recent years have witnessed the power of polynomial approximation methods in both algorithmic design and lower-bound constructions across various statistical and learning-theoretic topics. These include the application of Bernstein, Chebyshev polynomials, and some other good approximators to discrete distributions \cite{jiao2015minimax, wu2016minimax, orlitsky2016optimal, han2016minimax, han2018local, wu2019chebyshev, hao2020optimal, hao2020data}, the application of Hermite polynomials to Gaussian mixtures \cite{cai2010optimal, collier2017minimax, collier2020estimation, wu2020optimal, ma2022optimal, doss2023optimal, jung2026sharp, wang2026adaptive, chen2026sharp}, and other general topics \cite{ma2025best, polyanskiy2026dualizing}. Readers are referred to the monograph \cite{wu2020polynomial} for an overview. Our application of Laguerre polynomials to (transformed) Uniform-Beta mixtures adds a new recipe to this field.

\subsection{Paper Organization}
The remainder of this paper is organized as follows. In Section~\ref{sec:upper_bound_pivot} (resp. Section~\ref{sec:upper_bound_full}), we introduce our estimation procedure under pivotal reduction (resp. full observation) and present high-level ideas on how the performance guarantees are derived. In Section~\ref{sec:lower_bound}, information-theoretic lower bounds are provided to illustrate the worst-case (minimax) optimality of the aforementioned estimators. Section~\ref{sec:discussion} serves as a discussion of current limitations and future directions. Detailed proofs and auxiliary technical facts are deferred to the appendices.

\subsection{Notation}
We use the standard asymptotic notation, such as $O$ ($\lesssim$), $\Omega$ ($\gtrsim$), and $\Theta$ ($\asymp$). The subscripts in asymptotic notation denote the quantities on which hidden constants are allowed to depend. For example, $a=O_\delta(b)$ means that $a\leq C_\delta b$ for an absolute constant $C_\delta$ that depends only on $\delta$. The open and closed unit disks on the complex plane are denoted by
$\mathbb D=\{z\in\mathbb C: |z|<1\}$ and $\overline{\mathbb D}=\{z\in\mathbb C: |z|\le 1\}$, respectively. We use $\indi\{\cdot\}$ for the indicator function of an event and $\abs{\cdot}$ for the cardinality of a set. In light of the regularity condition \eqref{eq:not_one-hot}, we write
\begin{align}\label{eq:fixed_alphabet_ntp}
    \cP_\Delta^{\cW}\,\define \left\{P=(p_w)_{w\in\cW}\in [0,1]^{\abs{\cW}}:\ \sum_{w\in\cW}p_w=1,\ \max_{w\in\cW}p_w\leq 1-\Delta\right\}
\end{align}
as the class of $\Delta$-regular NTP distributions over a given token alphabet $\cW$, and 
\begin{align}\label{eq:unknown_alphabet_ntp}
    \cP_\Delta^{\geq 2}\,\define \left\{P=(p_j)_{j=1}^\infty\in [0,1]^{\infty}:\ 2\leq \abs{\{j:p_j>0\}}<\infty,\ \sum_{j\in \bbZ_+} p_j=1,\ \max_{j\in\bbZ_+ }p_j\leq 1-\Delta\right\}
\end{align}
as the class of $\Delta$-regular NTP distributions over unknown and arbitrary token alphabets whose cardinality is finite and at least $2$. The former notation \eqref{eq:fixed_alphabet_ntp} will be used as the family of NTP distributions under full observation, while the latter \eqref{eq:unknown_alphabet_ntp} appears under the pivotal reduction.

\section{Estimation under Pivotal Statistics Reduction}\label{sec:upper_bound_pivot}
In this section, we consider estimating $\varepsilon$ when only the pivotal statistics $\{Y_t\}_{t\in[N]}$ are observed. We show that an $\Theta_\delta((1/\tau)^{2\vee \Theta(1/\sqrt{\Delta})})$ sample complexity is attainable, that is, there exists an estimator that achieves an error bound of $O_\delta(N^{-\{(1/2)\wedge \Theta(\sqrt{\Delta})\}})$ using $N$ pivotal samples.
\begin{Theorem}\label{thm:upper_bound_pivotal}
    For any $\Delta\in(0,1)$ and $\delta\in(0,1)$, there exists an estimator $\widehat{\varepsilon}$ such that
    \begin{align*}
        \sup_{\varepsilon\in[0,1],\,\{P_t\}_{t\in[N]}\subset \cP_\Delta^{\geq 2}}\ P_{\{Y_t\}_{t\in[N]}}\left(\abs{\widehat{\varepsilon}-\varepsilon}> 1\wedge C \left(\frac{\log(e/\delta)}{N}\right)^{-\frac{2}{\pi}\arctan(\sqrt{\Delta})}\right)\leq \delta
    \end{align*}
    whenever $N\geq C'\log(e/\delta)$, where $C$ and $C'$ are universal absolute constants.
\end{Theorem}
Note that $(2/\pi)\arctan(\sqrt{\Delta})\asymp \sqrt{\Delta}$ for $h\in(0,1)$ and that $(2/\pi)\arctan(\sqrt{\Delta})<(2/\pi)\cdot (\pi/4)=1/2$. Therefore, Theorem~\ref{thm:upper_bound_pivotal} is essentially the $O_\delta(N^{-\{(1/2)\wedge\Theta(\sqrt{\Delta})\}})$ rate that we desire. 
We present the high-level idea of Theorem~\ref{thm:upper_bound_pivotal} in the remainder of this section, while detailed proofs are deferred to Appendix~\ref{apdx:proof_upper_bound_pivotal}.

The crucial idea of Theorem~\ref{thm:upper_bound_pivotal} is to leverage the connection between Laguerre polynomial bases and mixtures of exponential distributions, which are the negative logarithm of the Uniform-Beta mixtures~\eqref{eq:uniform_beta_mixture} that pivotal statistics follow. We will construct a polynomial-based estimator and show that it achieves the claimed rate. Let $L_m:(0,+\infty)\to \bbR$ be the degree-$m$ Laguerre polynomial, and define $\ell_m:(0,1)\to \bbR$ by $\ell_m(y)=L_m(-\log(y))$. Then, by change-of-variables in the integral,
\begin{align}\label{eq:ell_orthogonal}
    \int_{(0,1)}\ell_m(y)\ell_n(y)\dif y=\int_{(0,+\infty)}L_m(x)L_n(x)e^{-x}\dif x=\indi\{m=n\},
\end{align}
and thus $\{\ell_m\}_{m=0}^\infty$ forms a system of orthonormal polynomials of $-\log(y)$ on the interval $(0,1)$. Given a polynomial $f(x)=\sum_{m=0}^M a_m x^m$ with $a_{0:M}\subset \bbR$, we define its corresponding Laguerre filter $\ell_f:(0,1)\mapsto \bbR$ by 
\begin{align}\label{eq:laguerre_filter}
    \ell_f(y)\define \sum_{m=0}^M a_m\ell_m(y),\quad y\in(0,1).
\end{align}
By the orthogonality condition \eqref{eq:ell_orthogonal}, we have $\E[\ell_f(Y_t)]=a_0$ when $Y_t$ is not watermarked. More generally, in the hybrid setting, $\E[\ell_f(Y_t)]=(1-\varepsilon)a_0+\varepsilon\mathrm{Rem}(f,P_t)$ for some remainder term $\mathrm{Rem}(f,P_t)$ depending on both $f$ and $P_t$. Therefore, the value of $\varepsilon$ can be extracted from the mean of $\ell_f(Y_t)$ if we manage to select a polynomial $f$ that maintains decent $a_0\neq 0$ while uniformly suppressing the remainder for arbitrary $P_t\in\cP_\Delta^{\geq 2}$. In fact, we have:
\begin{Lemma}\label{lem:bias_variance}
    Let $f(x)=1+\sum_{m=1}^M a_m x^m$ be a real-coefficient polynomial and $\ell_f$ be as in \eqref{eq:laguerre_filter}. Then, for every $t\in[N]$, we have
    \begin{align*}
        \abs{\E[\ell_f(Y_t)]-(1-\varepsilon)}\leq \sup_{z\in [\Delta,1)}\abs{f(z)},\qquad \E[\ell_f^2(Y_t)]\leq \sup_{z\in\overline{\bbD}}\,\abs{f(z)}^2.
    \end{align*}
\end{Lemma}
See Appendix~\ref{apdx:proof_bias_variance} for the proof. Lemma~\ref{lem:bias_variance} prompts us to design an estimator by
\begin{align}
    \widetilde{\varepsilon}_f=1-\frac{1}{N}\sum_{t=1}^N\ell_f(Y_t),
\end{align}
where $f(x)=1+\sum_{m=1}^M a_m x^m$ is a real polynomial to be determined later. By Lemma~\ref{lem:bias_variance} and Chebyshev's inequality, we can control the bias--variance decomposition of $\widetilde{\varepsilon}_f$ as:
\begin{align*}
    \abs{\widetilde{\varepsilon}_f-\varepsilon}\leq \sup_{z\in[\Delta,1)}\abs{f(z)}+\frac{1}{\delta\sqrt{N}}\sup_{z\in\overline{\bbD}}\,\abs{f(z)}.
\end{align*}
Using another estimator $\widehat{\varepsilon}_f$ based on the median-of-means tournament\footnote{Divide the observations $\{Y_t\}_{t=1}^{N}$ into multiple blocks. Replace the sample mean of $\{\ell_f(Y_t)\}_{t\in[N]}$ used in \(\widetilde{\varepsilon}_f\) by the median of the block-wise means} \cite{devroye2016sub, lugosi2019mean}, we can sharpen the $(1/\delta)$-dependence and obtain a better result that
\begin{align*}
    \abs{\widehat{\varepsilon}_f-\varepsilon}\lesssim \sup_{z\in[\Delta,1)}\abs{f(z)}+\sqrt{\frac{\log(e/\delta)}{N}}\sup_{z\in\overline{\bbD}}\,\abs{f(z)}
\end{align*}
as long as $N\gtrsim \log(e/\delta)$. Then, it suffices to optimize over $f$ to find the best weighted balance between the slit-maximum $\sup_{z\in[\Delta,1)}\abs{f(z)}$ and the disk-maximum $\sup_{z\in\overline{\bbD}}\abs{f(z)}$, which can be tackled by the following result:
\begin{Lemma}\label{lem:polynomial_tradeoff}
For any $\Delta\in(0,1)$ and $\lambda\geq 1$, there exists a real polynomial $f(x)=1+\sum_{m=1}^M a_m x^m$ such that
\begin{align*}
    \sup_{z\in[\Delta,1)}\abs{f(z)}\leq 8e^{-2\lambda\kappa_\Delta},\quad \sup_{z\in\overline{\bbD}}\abs{f(z)}\leq 8e^{\lambda(1-2\kappa_\Delta)},
\end{align*}
with degree $M\leq 2\lambda^2+16\lambda\log(8\lambda)+1$, where $\kappa_\Delta=(2/\pi)\arctan(\sqrt{\Delta})$.
\end{Lemma}
See Appendix~\ref{apdx:proof_polynomial_tradeoff} for the proof. Lemma~\ref{lem:polynomial_tradeoff} further reduces the bias--variance tradeoff to
\begin{align*}
    \abs{\widehat{\varepsilon}_f-\varepsilon}\lesssim e^{-2\lambda\kappa_\Delta}+\sqrt{\frac{\log(e/\delta)}{N}}e^{\lambda(1-2\kappa_\Delta)},
\end{align*}
with the first term growing with $\lambda$ and the second decreasing with $\lambda$. Therefore, choosing $\lambda=\log(\sqrt{N/\log(e/\delta)})$ (which is $\geq 1$ when $N\gtrsim \log(e/\delta)$) to balance the two terms, we eventually obtain a rate of $O_\delta(n^{-(2/\pi)\arctan(\sqrt{\Delta})})=O_\delta(n^{-\{(1/2)\wedge \Theta(\sqrt{\Delta})\}})$.

The error exponent $(2/\pi)\arctan(\sqrt{\Delta})\asymp \sqrt{\Delta}$ originates from the delicate bias--variance trade-off between the slit-maximum and the disk-maximum of the polynomial $f$ as in Lemma~\ref{lem:polynomial_tradeoff}. In fact, such a tradeoff has already been order-wise optimal not only over real polynomials but also over all complex-valued functions that are holomorphic in $\bbD$ and continuous in $\overline{\bbD}$. See Section~\ref{sec:discussion} for more discussion.

\section{Estimation with Full Observation}\label{sec:upper_bound_full}
In this section, we consider the setting in which the full data \(\{(U_t,w_t)\}_{t\in[N]}\) are observed. It is intuitively true that a faster estimation rate is achievable in this setting than in the pivotal-only setting, since substantially more structures are observed. This intuition can be justified by an indicator-based estimator, which we will elaborate on below. 

Let $\beta_\Delta=\Delta/(1-\Delta)$. Consider the indicator function
\begin{align}
    B(u,w)=\indi\left\{u_w<\prod_{v\in\cW\backslash\{w\}}u_v^{1/\beta_\Delta}\right\},\quad (u,w)\in(0,1)^{\abs{\cW}}\times \cW.
\end{align}
The distribution of $B(U,w)$ contributes a signal gap of $\Delta^{\abs{\cW}-1}$ to separate watermarked and non-watermarked contents, as seen from the following lemma.
\begin{Lemma}\label{lem:different_bern}
    For any $P\in\cP^{\cW}_\Delta$, it holds that
    \begin{align*}
        B(U,w)\sim
        \begin{dcases}
            \mathrm{Bern}(\Delta^{\abs{\cW}-1}),\quad & U\sim \mathrm{Unif}(0,1)^{\otimes \abs{\cW}},\ w\sim P,\ w\perp U;\\
             \mathrm{Bern}(0),\quad &U\sim \mathrm{Unif}(0,1)^{\otimes \abs{\cW}},\ w\mid U\sim \delta_{\cS(U,P)}.
        \end{dcases}
    \end{align*}
\end{Lemma}
The event $\{u_w<\prod_{v\in\cW\backslash\{w\}}u_v^{1/\beta_\Delta}\}$ has a clean intuition.
For the watermarked contents, we know that $\log U_v \leq (p_v/p_w) \log U_w$ according to the Gumbel--max rule~\eqref{eq:gumbel_rule}.
Summing over all $v \neq w$ gives that $\log(\prod_{v\in\cW\backslash\{w\}}U_v) \leq [(1-p_w)/p_w] \log U_w \leq \beta_\Delta \log U_w$ and thus $B(U,w)=0$ for all $w$ selected from the Gumbel--max rule.
See Appendix~\ref{apdx:proof_different_bern} for the complete proof. 

According to Lemma~\ref{lem:different_bern}, we have
\begin{align*}
    \frac{1}{N}\sum_{t=1}^N B(U_t,w_t)\sim \frac{ (1-\varepsilon)}{N}\sum_{t=1}^NB_t,\quad B_t\iid \mathrm{Bern}(h^{\abs{\cW}-1}).
\end{align*}
Thus, we can estimate $\varepsilon$ by
\begin{align}\label{eq:estimator_full}
    \widehat{\varepsilon}=0\vee \left\{1-\frac{\sum_{t=1}^N B(U_t,w_t)}{N\Delta^{\abs{\cW}-1}}\right\}\wedge 1.
\end{align}
Note that the inner term is an unbiased estimator of $\varepsilon$. It is the summation of $N$ independent variables, each of which has an absolute value within \(1/(Nh^{\abs{\cW}-1})\), with total variance
\begin{align*}
    \frac{N\cdot \Delta^{\abs{\cW}-1}(1-\Delta^{\abs{\cW}-1})}{(N\Delta^{\abs{\cW}-1})^2}\leq \frac{1}{N\Delta^{\abs{\cW}-1}}.
\end{align*}
Therefore, by Bernstein's inequality, the inner term has an error within 
\begin{align*}
    \sqrt{\frac{2\log(2/\delta)}{N\Delta^{\abs{\cW}-1}}}+\frac{2\log(2/\delta)}{3N\Delta^{\abs{\cW}-1}} 
\end{align*}
with probability at least $1-\delta$. Finally, the truncation operation $0\vee\cdots\wedge 1$ ensures that the error is within $[0,1]$. In summary, we conclude that:
\begin{Theorem}
     Let $\widehat{\varepsilon}$ be the estimator as in \eqref{eq:estimator_full} based on full observation $\{(U_t,w_t)\}_{t\in [N]}$. For any $\cW$ with $\abs{\cW}\geq 2$, $\Delta\in(0,1-1/\abs{\cW}]$, and $\delta\in(0,1)$, we have
    \begin{align*}
        \sup_{\varepsilon\in[0,1],\, \{P_t\}_{t\in [N]}\subset\cP_\Delta^{\cW}}\ P_{\{(U_t,w_t)\}_{t\in[N]}}\left(\abs{\widehat{\varepsilon}-\varepsilon}> 1\wedge \sqrt{\frac{4\log(2/\delta)}{N\Delta^{\abs{\cW}-1}}}\right)\leq \delta.
    \end{align*}
\end{Theorem}

\section{Information-Theoretic Lower Bounds}\label{sec:lower_bound}
The current section is dedicated to proving information-theoretic lower bounds under full observation and pivotal reduction, respectively. To begin with, we claim that a trivial $\Omega_\delta(1/\tau^2)$ sample complexity ($\Omega_\delta(1/\sqrt{N})$ estimation error) is inevitable in both settings. It suffices to prove this for the more informative setup of full observation. In particular, we consider testing
\begin{align*}
    H_0:\ \varepsilon=\frac{1}{2}+\tau,\  p_{t,w}\equiv 1/\abs{\cW}\qquad \textnormal{versus}\qquad H_1:\ \varepsilon=\frac{1}{2},\ p_{t,w}\equiv 1/\abs{\cW},
\end{align*}
where the nuisance NTP distributions are stationary and flat. By Le Cam's two-point argument, any estimator must suffer an error proportional to $\tau=\tau(N)$ if the above pair of hypotheses forces the minimal sum of type-I and type-II errors to be at least $\Omega(\delta)$ under sample size $N$. In fact, the chi-square divergence between the two underlying distributions is $\chi^2(Q_{H_0},Q_{H_1})=4\tau^2(\abs{\cW}-1)/(\abs{\cW}+1)\leq 4\tau^2$. Therefore, a lower bound of $\Omega_\delta(1/\sqrt{N})$ can be obtained by setting $\chi^2(Q_{H_0},Q_{H_1})\lesssim_\delta 1/N$. However, this parametric lower bound is only useful when $\Delta=\Omega(1)$. When $\Delta$ is diminishing, a subtle mixture-versus-mixture construction can produce much sharper lower bounds.

\subsection{Lower Bound under Pivotal Statistics Reduction}\label{sec:lower_bound_pivot}
We first consider the pivot-based setting, in which we manage to show:
\begin{Theorem}\label{thm:lower_bound_pivotal}
    Consider estimating $\varepsilon$ based on pivotal statistics $\{Y_t\}_{t\in[N]}$. For any $\Delta\in(0,1)$ and $\delta\in(0,1)$, it holds that
    \begin{align*}
        \inf_{\widehat{\varepsilon}}\, \sup_{\varepsilon\in[0,1],\, \{P_t\}_{t\in[N]}\subset \cP_\Delta^{\geq 2}}\, P_{\{Y_t\}_{t\in[N]}}\left(\abs{\widehat{\varepsilon}-\varepsilon}\geq  \frac{1}{2}\wedge c_\delta N^{-\{(1/2)\wedge c\sqrt{\Delta}\}}\right)>\delta,
    \end{align*}
    for a universal constant $c>0$ and a constant $c_\delta>0$ that depends only on $\delta$.
\end{Theorem}
The improvement over the naive parametric lower bound $\Omega_\delta(1/\sqrt{N})$ is due to the use of mixtures over the nuisance NTP class rather than fixing a single sequence of $P_{1:N}$. We consider testing
\begin{align*}
    H_0:\ \varepsilon=\frac{1}{2}+\tau,\ P_{1:N}\iid \Lambda_0\qquad \textnormal{versus}\qquad H_1:\ \varepsilon=\frac{1}{2},\ P_{1:N}\iid \Lambda_1,
\end{align*}
where $\Lambda_0$ and $\Lambda_1$ are particularly chosen priors on $\cP_\Delta^{\geq 2}$ that hide lower-order information through moment matching.

The construction of $\Lambda_0$ and $\Lambda_1$ is as follows. For $P=(p_1,p_2,\cdots,p_J)\in \cP_\Delta^{\geq 2}$, we denote its corresponding purely watermarked pivotal density and CDF by
\begin{align*}
    g_P(r)=\sum_{j=1}^J r^{1/p_j-1},\quad G_P(r)=\sum_{j=1}^J p_j r^{1/p_j},\quad r\in[0,1].
\end{align*}
For a signed measure $\Sigma$ over $\cP_\Delta^{\geq 2}$, we use the notation $g_\Sigma$ (resp. $G_\Sigma$) to denote the mixture density $\int_{P\in\cP_\Delta^{\geq 2}} g_P\dif \Sigma(P)$ (resp. mixture CDF $\int_{P\in\cP_\Delta^{\geq 2}} G_P\dif \Sigma(P)$).
\begin{Proposition}\label{prop:matching_gumbel_pivotal}
    For any $\Delta\in(0,1/2)$ and $M\in \bbZ_+$, define $\rho_\Delta=\arccosh(\frac{1+2\Delta}{1-2\Delta})$. There exists a signed measure $\Sigma$ on $\cP_\Delta^{\geq 2}$ that satisfies
    \begin{enumerate}
        \item normalization:
        \begin{align*}
            \Sigma(\cP_\Delta^{\geq 2})=1,\quad 1\leq \norm{\Sigma}_{\mathrm{TV}}\leq 3\cosh\left(M\rho_\Delta\right);
        \end{align*}
        \item approximate uniformity:
        \begin{align*}
            \norm{g_\Sigma-1}_{L^2[0,1]}\leq 2^{-M} e^{M\rho_\Delta}.
        \end{align*}
    \end{enumerate}
\end{Proposition}
See Appendix~\ref{apdx:proof_matching_gumbel_pivotal} for the proof. With $\Sigma$ specified by Proposition~\ref{prop:matching_gumbel_pivotal}, we define  
\begin{align*}
    \Lambda_0=\frac{\abs{\Sigma}}{\norm{\Sigma}_\mathrm{TV}}=\frac{\Sigma^++\Sigma^-}{\norm{\Sigma}_\mathrm{TV}},\quad \Lambda_1=(1-2\tau)\Lambda_0+2\tau \Sigma.
\end{align*}
Then, $\Lambda_0$ is obviously a probability measure on $\cP_\Delta^{\geq 2}$, while $\Lambda_1$ satisfies
\begin{align*}
    \Lambda_1(\cP_\Delta^{\geq 2})=1-2\tau+2\tau=1,\quad \dif \Lambda_1=\left(\frac{1-2\tau}{\norm{\Sigma}_\mathrm{TV}}+2\tau\right)\dif \Sigma^++\left(\frac{1-2\tau}{\norm{\Sigma}_\mathrm{TV}}-2\tau\right)\dif \Sigma^-.
\end{align*}
Therefore, to make $\Lambda_1$ a genuine probability measure, we need
\begin{align*}
    (1-2\tau)/\norm{\Sigma}_\mathrm{TV}-2\tau\geq 0\quad \Longleftrightarrow\quad 0\leq \tau\leq \frac{1}{2+2\norm{\Sigma}_\mathrm{TV}},
\end{align*}
which is satisfied once $0\leq \tau\leq (1/8)e^{-M\rho_\Delta}$. Now, consider two distributions
\begin{align*}
    Q_0=(\frac{1}{2}+\tau) \mathrm{Unif}(0,1)+(\frac{1}{2}-\tau) G_{\Lambda_0},\quad Q_1=\frac{1}{2}\mathrm{Unif}(0,1)+\frac{1}{2}G_{\Lambda_1},
\end{align*}
and let $q_0$ and $q_1$ be their densities, respectively. We have
\begin{align*}
    q_1(r)-q_0(r)=-\tau+g_{\frac{1}{2}\Lambda_1-(\frac{1}{2}-\tau)\Lambda_0}=\tau (g_\Sigma-1).
\end{align*}
Meanwhile, we have $q_0(r)\geq 1/2$, and therefore,
\begin{align*}
    \chi^2(Q_1,Q_0)=\int\frac{(q_1-q_0)^2}{q_0}\leq 2\tau^2\norm{g_\Sigma-1}_{L^2[0,1]}^2\leq 2\tau^2 2^{-2M}e^{2M\rho_\Delta}.
\end{align*}
The indistinguishability between $Q_0^{\otimes N}$ and $Q_1^{\otimes N}$ (resulting an error lower bound $\tau$) is thus reduced to
\begin{align*}
    \tau^2 2^{-2M}e^{2M\rho_\Delta}\leq \frac{c_\delta}{N},\quad 0\leq \tau \leq (1/8)e^{-M\rho_\Delta}
\end{align*}
that is,
\begin{align*}
    \tau \lesssim \min\left\{e^{M(\log(2)-\rho_\Delta)}\sqrt{\frac{c_\delta}{N}},\,e^{-M\rho_\Delta}\right\}.
\end{align*}
When $0<\rho_\Delta\leq \log(2)$, that is, $0<\Delta\leq 1/18$, the first term is non-decreasing, while the second term is non-increasing. Setting $M=1\vee\floor{\log(c_\delta/N)/(2\log(2))}$ to balance these two terms, we maximize the right-hand side and obtain (recall also that $\tau\leq 1/2$)
\begin{align*}
    \tau\lesssim 1\wedge c_\delta N^{-\rho_\Delta/(2\log(2))}\asymp 1\wedge c_\delta N^{-\Theta(\sqrt{\Delta})}.
\end{align*}
For $\Delta\in(1/18,1)$, we use the trivial parametric lower bound $\Omega_\delta(N^{-1/2})$. Combining the two terms, we obtain a minimax lower bound of $1\wedge \Omega_\delta(N^{-\{(1/2)\wedge \Theta(\sqrt{\Delta})\}})$.

\subsection{Lower Bound for Estimating with Full Observation}\label{sec:lower_bound_full}
We then move on to the full observation setting, in which we present a lower bound of:
\begin{Theorem}\label{thm:lower_bound_full}
    Consider estimating $\varepsilon$ based on full observation $\{(U_t,w_t)\}_{t\in[N]}$. For any $\cW$ with $\abs{\cW}\geq 2$, any $\Delta\in(0,1-1/\abs{\cW})$, and $\delta\in(0,1)$, it holds that
    \begin{align*}
        \inf_{\widehat{\varepsilon}}\, \sup_{\varepsilon\in[0,1],\, \{P_t\}_{t\in[N]}\subset \cP_\Delta^{\cW}}\, P_{\{(U_t,w_t)\}_{t\in[N]}}\left(\abs{\widehat{\varepsilon}-\varepsilon}\geq c_{\delta,\abs{\cW}}\left(1\wedge\sqrt{\frac{1}{N\Delta^{\abs{\cW}-1}}}\right)\right)>\delta,
    \end{align*}
    for a constant $c_{\delta,\abs{\cW}}>0$ that depends only on $\delta$ and $\abs{\cW}$.
\end{Theorem}
Following the same spirit as in Section~\ref{sec:lower_bound_pivot}, we prove Theorem~\ref{thm:lower_bound_full} by constructing mixtures over the nuisance NTP class $\cP_\Delta^{\cW}$ and reducing the estimation lower bound to the hardness of testing a pair of composite hypotheses. 

However, since $\cW$ is known with a prescribed cardinality under full observation, the construction of mixtures becomes more restricted. We specify a modified construction of mixtures as follows. For $\Delta\in(0,1/2]$, let $\beta_\Delta=\Delta/(1-\Delta)\in(0,1]$. For each $w\in \cW$ and $\emptyset\neq S\subseteq \cW\backslash\{w\}$, define $P^{(w,S)}$ as a probability distribution over $\cW$ specified by
\begin{align}\label{eq:pws}
    P^{(w,S)}=(p^{(w,S)}_u)_{u\in \cW},\quad p^{(w,S)}_u=\begin{dcases}
        \frac{1}{1+\beta_\Delta\abs{S}},\quad &u=w;\\
        \frac{\beta_\Delta}{1+\beta_\Delta\abs{S}},\quad &u\in S;\\
        0,\quad &\textnormal{otherwise}.
    \end{dcases}
\end{align}
We have $P^{(w,S)}\in\cP_\Delta^{\cW}$ since
\begin{align*}
    \max_{u\in\cW} p^{(w,S)}_u=\frac{1}{1+\beta_\Delta\abs{S}}\leq \frac{1}{1+\beta_\Delta}=\Delta\leq 1-\Delta.
\end{align*}
Let 
\begin{align}\label{eq:index_set}
    \cI=\left\{(w,S):\ w\in\cW,\ \emptyset\neq S\subseteq \cW\backslash\{w\}\right\}.
\end{align}
We have $\abs{\cI}=\abs{\cW}(2^{\abs{\cW}-1}-1)$. Instead of putting mixing priors on entire NTP class $\cP_\Delta^{\cW}$, we construct priors $\Lambda_0$ and $\Lambda_1$ on the subset $\{P^{(w,S)}:\, (w,S)\in\cI\}$. Then, it suffices to find priors on the index set $\cI$ whose induced priors over $\{P^{(w,S)}:\, (w,S)\in\cI\}$ can almost hide the difference in watermark proportions. The detailed prior construction and the proof of Theorem~\ref{thm:lower_bound_full} are deferred to Appendix~\ref{apdx:proof_lower_bound_full}.

\section{Discussion}\label{sec:discussion}
Several discussions are in order.
\paragraph{Proportion estimation under other watermarking mechanisms.}
This paper considers the Gumbel--max watermarking. There are other watermarking mechanisms, including the green-red list \cite{kirchenbauer2023watermark} and the inverse transform \cite{kuditipudi2024robust}. As shown in \cite{li2025optimal}, the proportion of watermarks cannot be consistently estimated under the green-red list mechanism even with side information. The pivotal distribution under the inverse transform mechanism admits a closed form only in the infinite-sample asymptotic regime, thereby hindering the study of sample complexity. Consequently, there is no natural generalization from the Gumbel--max mechanisms to the others, and the proportion estimation problem under different watermarking mechanisms should plausibly be investigated case-by-case.

\paragraph{Optimality of Lemma~\ref{lem:polynomial_tradeoff}.}
In Section~\ref{sec:upper_bound_pivot}, we quantify the bias--variance trade-off through Lemma~\ref{lem:polynomial_tradeoff} which essentially states that:

\textit{For any $\Delta\in(0,1)$ and any absolute constant $A$ that is large enough, there exists a real polynomial $f$ with $f(0)=1$ such that $\sup_{z\in\overline{\bbD}}\, \abs{f(z)}\leq A$ and $\sup_{z\in[\Delta,1)}\abs{f(z)}\leq O(A^{-2\kappa_\Delta/(1-2\kappa_\Delta)})$.}

We claim that the above trade-off scaling is unavoidable even for all complex-valued functions that are holomorphic in $\bbD$ and continuous in $\overline{\bbD}$, as justified by the following result (proved in Appendix~\ref{apdx:proof_tradeoff_optimality}):
\begin{Claim}\label{clm:tradeoff_optimality}
    For any $A>0$ and $\Delta\in(0,1)$, for $f\in H(\bbD)\cap C(\overline{\bbD})$ that satisfies $f(0)=1$ and $\sup_{z\in\overline{\bbD}}\,\abs{f(z)}\leq M$, it must hold that $\sup_{z\in[\Delta,1)}\abs{f(z)}\geq A^{-2\kappa_\Delta/(1-2\kappa_\Delta)}$, where $\kappa_\Delta=(2/\pi)\arctan(\sqrt{\Delta})$.
\end{Claim}

\paragraph{Adaptivity.}
Our estimators in Section~\ref{sec:upper_bound_pivot} and Section~\ref{sec:upper_bound_full} require knowing the regularity parameter $\Delta$, but can be made adaptive to $\Delta\in[0,1]$ using Lepskii's method \cite{lepskii1991problem, lepskii1992asymptotically}. Take the pivot-based setting as an example. We can partition $[0,1]$ into grids of $\Delta_j\asymp j/\log^2(N)$, and run our estimator for each of these $\Delta_j$. Then, Lepskii's procedure can produce an aggregated estimator with rate $1\wedge C(\log(\log(N)/\delta)/N)^{\kappa_\Delta}$, only incurring iterated logarithmic adaptation cost. The full observation estimator can also be made adaptive by running Lepskii's method on modified grids.

\paragraph{Precise optimal sample complexities.}
In the setting of pivotal reduction, both our upper and lower bounds of sample complexity read $\Theta_\delta((1/\tau)^{2\vee\Theta(1/\sqrt{\Delta})})$. However, the $\Theta(1/\sqrt{\Delta})$ term on the exponent may contain different leading constants for upper and lower bounds, making them not of the same order. This is a natural drawback of the polynomial method, as also noted in the results of \cite{orlitsky2016optimal,wu2019chebyshev, hao2020optimal}. In addition, the upper and lower bounds for full observation differ by a $\abs{\cW}$-dependent term when the size of the alphabet can grow with $N$. To derive the most precise optimal sample complexities, one must carefully pin down the choice of polynomials in both the design of estimators and the construction of lower bound instances.

\bibliographystyle{alpha}
\bibliography{reference}

\appendix
\section{Remaining Proof of Upper Bounds}\label{apdx:proof_upper_bound}

\subsection{Proof of Lemma~\ref{lem:bias_variance}}\label{apdx:proof_bias_variance}
\begin{proof}[Proof of Lemma~\ref{lem:bias_variance}]
We first control the bias. Using Lemma~\ref{lem:laguerre_laplace}, we have
\begin{align*}
    \E[\ell_f(Y_t)]&=1+\sum_{m=1}^M a_m \E[\ell_m(Y_t)]\\
    &=1+\sum_{m=1}^M a_m \int_{(0,1)}\left[(1-\varepsilon)+\varepsilon \sum_{w\in\cW} r^{1/p_{t,w}}\right] L_m(-\log(r))\dif r\\
    &=1+\sum_{m=1}^Ma_m\int_{(0,+\infty)}\left[(1-\varepsilon)+\varepsilon\sum_{w:p_{t,w}>0}e^{-x/p_{t,w}}\right]L_m(x)\dif x\\
    &=1+\varepsilon\sum_{m=1}^Ma_m\sum_{w:p_{t,w}>0}\int_{(0,+\infty)}e^{-x/p_{t,w}}L_m(x)\dif x \\
    &=1+\varepsilon\sum_{m=1}^Ma_m\sum_{w:p_{t,w}>0}\frac{(1/p_{t,w}-1)^m}{(1/p_{t,w})^{m+1}}\\
    &=1+\varepsilon\sum_{m=1}^Ma_m\sum_{w:p_{t,w}>0} p_{t,w}(1-p_{t,w})^m\\
    &=1-\varepsilon+\varepsilon\sum_{w:p_{t,w}>0}p_{t,w}f(1-p_{t,w}).
\end{align*}
Since $1-p_{t,w}\in [\Delta,1)$ for those $p_{t,w}>0$, we obtain that
\begin{align*}
    \abs{\E[1-\ell_f(Y_t)]-\varepsilon}=\varepsilon\abs*{\sum_{w:p_{t,w}>0}p_{t,w}f(1-p_{t,w})}\leq \sup_{z\in[\Delta,1)}\abs{f(z)}.
\end{align*}

We move on to control the second moment. For $p\in(0,1]$, define the Gram matrix (operator)
\begin{align*}
    G_{j,k}(p)=\E_{Y\sim \mathrm{Beta}(1/p,1)}[\ell_j(Y)\ell_k(Y)].
\end{align*}
By Lemma~\ref{lem:laguerre_generating_function}, we have
\begin{align*}
    \sum_{j,k\geq 0}G_{j,k}(p)u^jv^k&=\int (1/p)e^{-y/p}\frac{e^{-uy(1-u)-vy(1-v)}}{(1-u)(1-v)}\dif y\\
    &=\frac{1}{1-(1-p)(u+v)+(1-2p)uv},\quad u,v\in(-1,1).
\end{align*}
Let $\zeta_p(z)=1-p+pz$. Define the kernel
\begin{align*}
    H_{j,k}(p)=\frac{1}{2\pi}\int_0^{2\pi}\zeta_p(e^{\im\theta})^j\overline{\zeta_p(e^{\im \theta})^k}\dif \theta.
\end{align*}
We compute its generating function
\begin{align*}
    \sum_{j,k\geq 0}H_{j,k}(p)u^jv^k&=\frac{1}{2\pi}\int_0^{2\pi}\frac{1}{1-u\zeta_p(e^{\im \theta})}\cdot \frac{1}{1-v\overline{\zeta_p(e^{\im \theta})}}\dif \theta\\
    &=\frac{1}{1-(1-p)(u+v)+(1-2p)uv},\quad \abs{u}<1,\ \abs{v}<1,
\end{align*}
where the last line follows from a direct evaluation of the contour integral. Therefore, we have $G_{j,k}(p)\equiv H_{j,k}(p)$, and thus for every $p\in(0,1]$,
\begin{align*}
    \E_{Y\sim\mathrm{Beta}(1/p,1)}[\ell_f^2(Y)]=\sum_{j,k=0}^M a_j \overline{a_k}G_{j,k}(p)&=\sum_{j,k=0}^M a_j \overline{a_k}H_{j,k}(p)\\
    &=\frac{1}{2\pi}\int_0^{2\pi}\sum_{j,k=0}^M a_j(1-p+pe^{\im\theta})^j\overline{a_k(1-p+pe^{\im\theta})^k}\dif \theta\\
    &=\frac{1}{2\pi}\int_0^{2\pi}\abs{f(1-p+pe^{\im\theta})}^2\dif \theta\\
    &\leq \sup_{z\in\overline{\bbD}}\,\abs{f(z)}^2.
\end{align*}
Consequently,
\begin{align*}
    \E[\ell_f^2(Y_t)]&=(1-\varepsilon)\E_{Y\sim\mathrm{Beta}(1,1)}[\ell_f^2(Y)]+\varepsilon \sum_{w\in\cW:p_{t,w}>0} p_{t,w} \E_{Y\sim\mathrm{Beta}(1/p_{t,w},1)}[\ell_f^2(Y)]\\
    &\leq(1-\varepsilon)\sup_{z\in\overline{\bbD}}\,\abs{f(z)}^2+\varepsilon \left(\sum_{w\in\cW:p_{t,w}>0}p_{t,w}\right)\cdot \sup_{z\in\overline{\bbD}}\,\abs{f(z)}^2\\
    &\leq \sup_{z\in\overline{\bbD}}\,\abs{f(z)}^2.
\end{align*}
\end{proof}
\subsection{Proof of Lemma~\ref{lem:polynomial_tradeoff}}\label{apdx:proof_polynomial_tradeoff}
The proof of Lemma~\ref{lem:polynomial_tradeoff} is divided into two parts. We first find a closed-form analytic function $g_{\lambda,a}$ in $\bbD$ that satisfies similar conditions. Then, we choose $f$ as the degree-$M$ truncated Taylor series of dilated $g_{\lambda,a}$ for appropriate values of $(M,a)$ depending on $\lambda$ and $\Delta$. To begin with, for $a\in(0,1)$ and $\lambda\geq 1$, we define
\begin{align*}
    \xi_a(z)=\frac{z-a}{1-az},\quad B_\lambda(z)=\cos\left((4\lambda/\pi)\arctanh(z)\right)
\end{align*}
for $z\in\bbD$. Here, the complex function $\arctanh(\cdot)$ is defined through
\begin{align*}
    \arctanh(z)=\sum_{k=0}^\infty\frac{z^k}{2k+1},\quad z\in \bbD.
\end{align*}
Since $B_\lambda(z)=B_\lambda(-z)$, we can expand the holomorphic function $B_\lambda$ without odd-order terms:
\begin{align*}
    B_\lambda(z)=\sum_{k=0}^\infty b_{k} z^{2k},\quad z\in\bbD.
\end{align*}
Also, since $B_\lambda(\overline{z})=\overline{B_\lambda(z)}$, all coefficients $\{b_k\}_{k=0}^\infty$ are real. We can then define another holomorphic function $D_\lambda$ on $\bbD$ by
\begin{align*}
    D_\lambda(z)=\sum_{k=0}^\infty b_{k}  z^k,\quad z\in\bbD.
\end{align*}
This function $D_\lambda$ is also well-defined on $\bbD$ since for each $z\in\bbD$, there exists $w\in\bbD$ such that $w^2=z$ and thus $D_\lambda(z)=B_\lambda(w)$. We then define $g_{\lambda,a}$ by
\begin{align*}
    g_{\lambda,a}(z)=\frac{D_\lambda(\xi_a(z))}{D_\lambda(-a)},\quad z\in\bbD.
\end{align*}
Since $\xi_a$ is a disk automorphism, $g_{\lambda,a}$ is well-defined as a single-valued complex function. Recall that $\{b_k\}_{k=0}^\infty$ are all real, and thus
\begin{align*}
    g_{\lambda,a}(\overline{z})=\frac{D_\lambda(\xi_a(\overline{z}))}{D_\lambda(-a)}=\frac{D_\lambda(\overline{\xi_a(z)})}{D_\lambda(-a)}=\frac{\overline{D_\lambda(\xi_a(z))}}{D_\lambda(-a)}=\overline{g_{\lambda,a}(z)}.
\end{align*}
Therefore, the Taylor coefficients of $g_{\lambda,a}$ are also all real. We claim that:
\begin{Lemma}\label{lem:helper_function}
    For $a\in(0,1)$, $\lambda\geq 1$, and $g_{\lambda,a}$ defined as above, we have
    \begin{align*}
        g_{\lambda,a}(0)=1,\quad \sup_{z\in[a,1)}\abs{g_{\lambda,a}(z)}\leq 2e^{-2\lambda \kappa_a},\quad \sup_{z\in\bbD}\,\abs{g_{\lambda,a}(z)}\leq 2e^{-\lambda(1-2\kappa_a)}, 
    \end{align*}
    where $\kappa_a=(2/\pi)\arctan(\sqrt{a})$ is defined in analogy to $\kappa_\Delta$.
\end{Lemma}
\begin{proof}
First, it is straightforward that $\xi_a(0)=-a$ and thus $g_{\lambda,a}(0)=1$. When $z\in[a,1)$, $\xi_a(z)\in[0,1)$ is real and so is $\arctanh(\xi_a(z))$, and thus
\begin{align*}
    g_{\lambda,a}(z)=\frac{D_\lambda(\xi_a(z))}{D_\lambda(-a)}&=\frac{B_\lambda(\sqrt{\xi_a(z)})}{B_\lambda(\im \sqrt{a})}\\
    &=\frac{\cos((4\lambda/\pi)\arctanh(\sqrt{\xi_a(z)}))}{\cos((4\lambda/\pi)\arctanh(\im \sqrt{a}))}\\
    &\leq \frac{1}{\cos((4\lambda/\pi)\arctanh(\im \sqrt{a}))}\\
    &=\frac{1}{\cos(\im (4\lambda/\pi)\arctan(\sqrt{a}))}\\
    &=\frac{1}{\cosh(2\lambda \kappa_a)}\\
    &\leq 2e^{-2\lambda \kappa_a}.
\end{align*}
It remains to bound the disk supremum. When $z\in\bbD$, $\zeta=(1+z)/(1-z)$ has positive real part. Hence, we can write
\begin{align*}
    \arctanh(z)=\frac{1}{2}\Log\left(\frac{1+z}{1-z}\right),\quad z\in\bbD,
\end{align*}
where $\Log$ is the principal branch defined as $\Log(\zeta)=\log(\abs{\zeta})+\im \Arg(\zeta)$ with $\Arg(\zeta)\in(-\pi/2,\pi/2)$ for every $\zeta:\Re\zeta>0$. Therefore, for $z\in\bbD$, we have
\begin{align*}
    \abs{\Im \arctanh(z)}=\frac{1}{2}\Arg\left(\frac{1+z}{1-z}\right)\leq \frac{\pi}{4}.
\end{align*}
Consequently, 
\begin{align*}
    \sup_{z\in\bbD}\, \abs{B_\lambda(z)}=\sup_{z\in\bbD}\,\abs{\cos((4\lambda/\pi)\arctanh(z))}=\sup_{z\in\bbD}\, \abs{\cosh((4\lambda/\pi)\cdot \Im \arctanh(z))}\leq \cosh(\lambda).
\end{align*}
Again, by the fact that every $z\in\bbD$ can be written as $z=w^2$ for some $w\in\bbD$, we conclude that
\begin{align*}
    \sup_{z\in\bbD}\, \abs{D_\lambda(z)}\leq \sup_{z\in\bbD}\, \abs{B_\lambda(z)}\leq \cosh(\lambda),
\end{align*}
which eventually leads to
\begin{align*}
    \sup_{z\in\bbD}\,\abs{g_{\lambda,a}(z)}\leq \frac{\cosh(\lambda)}{\cosh(2\lambda\kappa_a)}\leq 2e^{\lambda(1-2\kappa_a)}.
\end{align*}
\end{proof}
We can now prove Lemma~\ref{lem:polynomial_tradeoff}.
\begin{proof}[Proof of Lemma~\ref{lem:polynomial_tradeoff}]
For $\Delta\in(0,1)$ and $\lambda\geq 1$, let 
\begin{align*}
    a=\Delta\left(1-\frac{1}{2\lambda}\right),\quad \rho=1-\frac{1}{2\lambda}.
\end{align*}
Then, $1/2\leq \rho<1$ and $0<\Delta/2\leq a=\rho \Delta<\Delta$. We define $f_{\lambda,\Delta}(z)=g_{\lambda,a}(\rho z)$ for $g_{\lambda,a}$ as in Lemma~\ref{lem:helper_function}. Since $g_{\lambda,a}$ is holomorphic in $\bbD$, $f_{\lambda,\Delta}$ is holomorphic in the dilated disk $\bbD(0,1/\rho)\define \{z:\, \abs{z}<1/\rho\}$. Therefore, we can write
\begin{align*}
    f_{\lambda,\Delta}(z)=\sum_{m=0}^\infty \frac{f_{\lambda,\Delta}^{(m)}(0)}{m!}z^m,\quad z\in\bbD(0,1/\rho).
\end{align*}
Meanwhile, we use $z\in\bbD(0,1/\rho)\Leftrightarrow \rho z\in\bbD$ to see that
\begin{align*}
    \sup_{z\in\bbD(0,1/\rho)}\, \abs{f_{\lambda,\Delta}(z)}=\sup_{z\in\bbD}\, \abs{g_{\lambda,a}(z)}\leq 2e^{\lambda(1-2\kappa_a)},
\end{align*}
and use $z\in [\Delta,1)\Leftrightarrow\rho z\in [a,1)$ to conclude that
\begin{align*}
    \sup_{z\in [\Delta,1)}\abs{f_{\lambda,\Delta}(z)}\leq \sup_{z\in [a,1)}\abs{g_{\lambda,a}(z)}\leq 2e^{-2\lambda \kappa_a}.
\end{align*}
By Cauchy's estimate, we can further claim that
\begin{align}\label{eq:cauchy_estimate}
    \abs{f_{\lambda,\Delta}^{(m)}(0)}\leq \frac{m!\sup_{z\in\bbD(1/\rho)}\abs{f_{\lambda,\Delta}(z)}}{(1/\rho)^m}\leq m!\rho^m\cdot 2e^{\lambda(1-2\kappa_a)}.
\end{align}
For $M\in\bbN$, let $f_{\lambda,\Delta,M}$ be the order-$M$ Taylor truncation of $f_{\lambda,\Delta}$, that is,
\begin{align*}
    f_{\lambda,\Delta,M}(z)=\sum_{m=0}^M \frac{f_{\lambda,\Delta}^{(m)}(0)}{m!}z^m,\quad z\in\bbC.
\end{align*}
Recall that $g_{\lambda,a}$ has real Taylor coefficients, and so does $f_{\lambda,\Delta}$. Therefore, $f_{\lambda,\Delta,M}$ is indeed a finite-degree polynomial with real coefficients. We further check that
\begin{align*}
    f_{\lambda,\Delta,M}(0)=f_{\lambda,\Delta}(0)=g_{\lambda,a}(0)=1,
\end{align*}
meaning that the constant term is indeed $1$. Also, using the previous Cauchy's estimate \eqref{eq:cauchy_estimate}, we have for every $z\in\overline{\bbD}$ that
\begin{align*}
    \abs{f_{\lambda,\Delta,M}(z)-f_{\lambda,\Delta}(z)}\leq \sum_{m\geq M+1}\frac{\abs{f_{\lambda,\Delta}^{(m)}(0)}}{m!}\leq 2e^{\lambda(1-2\kappa_a)}\sum_{m\geq M+1}\rho^m\leq 2e^{\lambda(1-2\kappa_a)}\frac{\rho^{M+1}}{1-\rho}.
\end{align*}
Choose 
\begin{align}
    M=\ceil*{\frac{\lambda+\log(8\lambda)}{-\log(1-1/(2\lambda))}},
\end{align}
which yields
\begin{align*}
    \frac{\rho^{M+1}}{1-\rho}=\frac{(1-1/(2\lambda))^{M+1}}{1-(1-1/(2\lambda))}\leq \frac{e^{-\lambda-\log(8\lambda)}}{1/(2\lambda)}=\frac{1}{4}e^{-\lambda}.
\end{align*}
Therefore, with this choice of $M$, we can achieve
\begin{align*}
    \sup_{z\in\bbD}\, \abs{f_{\lambda,\Delta,M}(z)-f_{\lambda,\Delta}(z)}\leq \frac{1}{4}e^{-\lambda}.
\end{align*}
Therefore, by the triangle inequality and that $\kappa_a\in(0,1/2)$,
\begin{align*}
    \sup_{z\in\overline{\bbD}}\, \abs{f_{\lambda,\Delta,M}(z)}\leq 2e^{\lambda(1-2\kappa_a)}+\frac{1}{4}e^{-\lambda}\leq \frac{9}{4}e^{\lambda(1-2\kappa_a)},
\end{align*}
and
\begin{align*}
     \sup_{z\in [\Delta,1)}\abs{f_{\lambda,\Delta,M}(z)}\leq 2e^{-2\lambda \kappa_a}+\frac{1}{4}e^{-\lambda}\leq \frac{9}{4}e^{-2\lambda\kappa_a}.
\end{align*}
To translate from $e^{-2\lambda\kappa_a}$ to $e^{-\lambda\kappa_\Delta}$, we notice that
\begin{align*}
    \frac{e^{-2\lambda\kappa_a}}{e^{-2\lambda\kappa_\Delta}}=e^{2\lambda (\kappa_\Delta-\kappa_a)}&= \exp\left((4\lambda/\pi)(\arctan(\sqrt{\Delta})-\arctan(\sqrt{a}))\right)\\
    &\leq \exp\left((4\lambda/\pi)\cdot (\Delta-a)\cdot \sup_{x\in[a,\Delta]}\frac{1}{2\sqrt{x}(1+x)}\right)\\
    &= \exp\left((2\Delta/\pi)\cdot \frac{1}{\sqrt{a}(1+a)}\right)\\
    &\leq \exp\left((2\Delta/\pi)\cdot \frac{1}{\sqrt{h/2}}\right)\\
    &\leq e^{2\sqrt{2}/\pi}\leq 3.
\end{align*}
The desired conclusions then follow.
\end{proof}

\subsection{Proof of Theorem~\ref{thm:upper_bound_pivotal}}\label{apdx:proof_upper_bound_pivotal}
To prove Theorem~\ref{thm:upper_bound_pivotal}, we first state the following performance guarantee for the median-of-means estimator.
\begin{Lemma}[Median-of-means]\label{lem:median_of_means}
    Let $\delta\in(0,1)$ and $N\geq 24\log(e/\delta)$. Suppose $X_{1:N}$ are independent random variables such that
    \begin{align*}
        \abs{\E[X_t]-\theta}\leq B,\quad \E[X_t^2]\leq V^2,\quad t\in[N],
    \end{align*}
    for some $\theta\in\bbR$, $B\geq 0$, and $V^2\geq 0$. Let $K=\ceil{8\log(e/\delta)}$ and $q=\floor{N/K}$, and divide the index set $[Kq]$ into $K$ equal-sized blocks $\{S_k=\{(k-1)q+1,(k-1)q+2,\cdots,kq\}\}_{k=1}^K$. Define 
    \begin{align*}
        Z_k=\frac{1}{q}\sum_{t\in B_k} X_t,\quad \widehat{\theta}=\mathrm{median}\{Z_1,Z_2,\cdots,Z_k\}.
    \end{align*}
    Then,
    \begin{align*}
        P_{X_{1:N}}\left(\abs{\widehat{\theta}-\theta}>B+6\sqrt{2}V\sqrt{\frac{\log(e/\delta)}{N}}\right)\leq \delta.
    \end{align*}
\end{Lemma}
\begin{proof}
    By Chebyshev's inequality,
    \begin{align*}
        P_{X_{1:N}}\left(\abs{Z_k-\E[Z_k]}>2V/\sqrt{q}\right)\leq \frac{1}{4}.
    \end{align*}
    Also,
    \begin{align*}
        \abs{\E[Z_k]-\theta}\leq \frac{1}{q}\sum_{t\in B_k}\abs{E[X_t]-\theta}\leq B.
    \end{align*}
    Therefore, by the triangle inequality, we have
    \begin{align*}
         P_{X_{1:N}}\left(\abs{Z_k-\theta}>B+2V/\sqrt{q}\right)\leq \frac{1}{4}.
    \end{align*}
    Let $I_k=\indi\{\abs{Z_k-\theta}>B+2V/\sqrt{q}\}$ be the indicator of this bad event. Then, $I_{1:K}$ are independent and $\E[I_k]\leq 1/4$ for all $k\in [K]$. If $\abs{\widehat{\theta}-\theta}>B+2V/\sqrt{q}$, then $I_k=1$ for at least $\ceil{K/2}$ blocks. Therefore,
    \begin{align*}
        P_{X_{1:N}}\left(\abs{\widehat{\theta}-\theta}>B+2V/\sqrt{q}\right)\leq P_{X_{1:N}}\left(\sum_{k=1}^K I_k\geq K/2\right)&\leq \exp\left(-\frac{2(K/2-1/4)^2}{K}\right)\\
        &\leq \exp\left(-K/8\right)\\
        &\leq \delta,
    \end{align*}
    where we have used Hoeffding's inequality and $K=\ceil{8\log(e/\delta)}>1$. Finally, we notice that
    \begin{align*}
        B+2V/\sqrt{q}=B+\frac{2V}{\sqrt{\floor{N/K}}}\leq B+2V\sqrt{\frac{2\ceil{8\log(e/\delta)}}{N}}\leq B+6\sqrt{2}V\sqrt{\frac{\log(e/\delta)}{N}},
    \end{align*}
    which recovers the desired bound.
\end{proof}
\begin{proof}[Proof of Theorem~\ref{thm:upper_bound_pivotal}] We set $X_t=\ell_f(Y_t)$, $B=\sup_{z\in[\Delta,1)}\abs{f(z)}$, and $V^2=\sup_{z\in\overline{\bbD}}\abs{f(z)}^2$ in Lemma~\ref{lem:median_of_means}. Then, according to Lemma~\ref{lem:bias_variance}, the resulting median-of-means estimator
\begin{align*}
    \widehat{\theta}=\mathrm{median}\left\{\frac{1}{q}\sum_{t\in B_k}\ell_f(Y_t):\ k\in [K]\right\}
\end{align*}
can estimate $\theta=1-\varepsilon$ up to an error of
\begin{align*}
    \sup_{z\in[\Delta,1)}\abs{f(z)}+6\sqrt{2}\sup_{z\in\overline{\bbD}}\abs{f(z)}\sqrt{\frac{\log(e/\delta)}{N}},
\end{align*}
with probability at least $1-\delta$. Let
\begin{align*}
    \widehat{\varepsilon}\define 0\vee(1-\widehat{\theta})\wedge 1=0\vee\left(1-\mathrm{median}\left\{\frac{1}{q}\sum_{t\in B_k}\ell_f(Y_t):\ k\in [K]\right\}\right)\wedge 1.
\end{align*}
Then, we have whenever $N\geq 24\log(e/\delta)$ that
\begin{align*}
    P_{Y_{1:N}}\left(\abs{\widehat{\varepsilon}-\varepsilon}>1\wedge \left(\sup_{z\in[\Delta,1)}\abs{f(z)}+6\sqrt{2}\sup_{z\in\overline{\bbD}}\abs{f(z)}\sqrt{\frac{\log(e/\delta)}{N}}\right)\right)\leq \delta.
\end{align*}
According to Lemma~\ref{lem:polynomial_tradeoff}, for every $\lambda\geq 1$, we can choose a real polynomial $f$ to make the error term satisfy
\begin{align*}
    \sup_{z\in[\Delta,1)}\abs{f(z)}+6\sqrt{2}\sup_{z\in\overline{\bbD}}\abs{f(z)}\sqrt{\frac{\log(e/\delta)}{N}}\leq 8e^{-2\lambda\kappa_\Delta}+48\sqrt{2}e^{\lambda(1-2\kappa_\Delta)}\sqrt{\frac{\log(e/\delta)}{N}}.
\end{align*}
Taking $\lambda=(1/2)\log(N/\log(e/\delta))\geq \log(24)/2>1$, we have $e^{\lambda}=\sqrt{N/\log(e/\delta)}$, and the upper bound becomes
\begin{align*}
    1\wedge (8+48\sqrt{2})\left(\frac{\log(e/\delta)}{N}\right)^{\kappa_\Delta}.
\end{align*}
\end{proof}

\subsection{Proof of Lemma~\ref{lem:different_bern}}\label{apdx:proof_different_bern}
\begin{proof}[Proof of Lemma~\ref{lem:different_bern}]
When $(U,w)$ is sampled under the Gumbel--max mechanism with NTP distribution $P\in\cP_\Delta^{\cW}$, we have
\begin{align*}
    U_{w}^{1/p_w}\geq U_v^{1/p_v},\quad \forall \, v\in\cW\backslash\{v\}.
\end{align*}
That is, for every $v\in\cW\backslash\{w\}$, \(  U_w^{p_v}\geq U_v^{p_w}\), and thus
\begin{align*}
    U_w^{1-p_w}=\prod_{v\neq w}U_w^{p_v}\geq \prod_{v\neq w}U_v^{p_w}.
\end{align*}
Therefore, we use $p_w\leq 1-\Delta$ to claim that
\begin{align*}
    U_w\geq \prod_{v\neq w} U_v^{p_w/(1-p_w)}\geq \prod_{v\neq w} U_v^{(1-\Delta)/\Delta}=\prod_{v\neq w} U_v^{1/\beta_\Delta},
\end{align*}
which means $B(U,w)=0$ almost surely in the purely watermarked setting. To show the conclusion for purely non-watermarked tokens, we use the independence between $w$ and $U$ to claim that
\begin{align*}
    P(B(U,w)=1)&=\E_{U_{1:\abs{\cW}}\iid \mathrm{Unif}(0,1)}\left(\indi\left\{U_1<\prod_{2\leq j\leq \abs{\cW}} U_j^{1/\beta_\Delta}\right\}\right)\\
    &=\E_{U_{2:\abs{\cW}}\iid \mathrm{Unif}(0,1)}\left[\E_{U_1\sim  \mathrm{Unif}(0,1)}\left(\indi\left\{U_1<\prod_{2\leq j\leq \abs{\cW}} U_j^{1/\beta_\Delta}\right\}\mid U_{2:\abs{\cW}}\right)\right]\\
    &=\E_{U_{2:\abs{\cW}}\iid \mathrm{Unif}(0,1)}\left[\prod_{2\leq j\leq \abs{\cW}} U_j^{1/\beta_\Delta}\right]\\
    &=\prod_{2\leq j\leq \abs{\cW}} \left(\int_{(0,1)} u^{1/\beta_\Delta}\dif u\right)\\
    &=\Delta^{\abs{\cW}-1}.
\end{align*}
The desired conclusion then follows.
\end{proof}

\subsection{Proof of Claim~\ref{clm:tradeoff_optimality}}\label{apdx:proof_tradeoff_optimality}
\begin{proof}[Proof of Claim~\ref{clm:tradeoff_optimality}]
For any polynomial $f$ such that $f(0)=1$ and $\sup_{z\in\overline{\bbD}}\,\abs{f(z)}\leq A$, let 
    \begin{align*}
        a=\sup_{z\in[\Delta,1)}\abs{f(z)}.
    \end{align*}
    We claim $a\in(0,A]$. The reason for $a>0$ is as follows: if this is not true, we would have $f(z)\equiv 0$ on $[\Delta,1]$, which implies $f(z)\equiv 0$ in $\bbD$ since $f\in H(\bbD)$. This contradicts our assumption $f(0)=1\neq 0$. 

    Define the conformal mapping:
    \begin{align*}
        w\mapsto z=\Psi(w)\define \frac{(\frac{e^{\im \pi w/2}-1}{e^{\im\pi w/2}+1})^2+\Delta}{1+\Delta(\frac{e^{\im \pi w/2}-1}{e^{\im\pi w/2}+1})^2}.
    \end{align*}
    Then, we have
    \begin{align*}
        w\in\{\zeta:0<\Re \zeta<1\}\Leftrightarrow z\in \bbD\backslash[\Delta,1),
    \end{align*}
    and also
    \begin{align*}
        w\in\{\zeta:\Re \zeta=0\}\Leftrightarrow z\in[\Delta,1),\quad w\in\{\zeta:\Re \zeta =1\}\Leftrightarrow z\in \partial\bbD\backslash\{0\},\quad w=\infty \Leftrightarrow z=1,
    \end{align*}
    and finally,
    \begin{align*}
        w=\frac{4}{\pi}\arctan(\sqrt{\Delta})\define 2\kappa_\Delta\Leftrightarrow z=0.
    \end{align*}
    Since $f\in H(\bbD)\cap C(\overline{\bbD})$, the composite function $g(w)=f\circ \Psi(w)$ is a bounded holomorphic function on the strip $\{\zeta:0<\Re \zeta<1\}$, and is continuous on its closure. Define
    \begin{align*}
        G(x)=\sup_{y\in\bbR }\, \abs{g(x+\im y)},\quad 0\leq x\leq 1.
    \end{align*}
    By assumption, we have
    \begin{align*}
        G(0)\leq a,\quad G(1)\leq A,\quad G(2\kappa_\Delta)\geq 1.
    \end{align*}
    According to Hadamard three-lines theorem (Lemma \ref{lem:hadamard}), we have
    \begin{align*}
        G(2\kappa_\Delta)\leq G(0)^{1-2\kappa_\Delta}G(1)^{2\kappa_\Delta},
    \end{align*}
    which implies
    \begin{align*}
        1\leq a^{1-2\kappa_\Delta}A^{2\kappa_\Delta}.
    \end{align*}
    Rearranging the terms yields the desired lower bound.
    
\end{proof}

\section{Remaining Proof of Lower Bounds}\label{apdx:proof_lower_bound}
\subsection{Proof of Proposition~\ref{prop:matching_gumbel_pivotal}}\label{apdx:proof_matching_gumbel_pivotal}
We first state the following moment-matching result on a bounded interval.
\begin{Lemma}\label{lem:matching_plain}
For any $\Delta\in(0,1/2)$, any positive integer $M$, and $\rho_\Delta=\arccosh(\frac{1+2h}{1-2h})$, there exists a signed measure $\sigma$ on $[\Delta,1/2]$ supported on at most $2M+4$ points that satisfies
\begin{enumerate}
    \item normalization:
    \begin{align*}
        \sigma([\Delta,1/2])=1,\quad \norm{\sigma}_\mathrm{TV}\leq \cosh(M\rho_\Delta);
    \end{align*}
    \item zero moments:
    \begin{align*}
        \int_{[\Delta,1/2]} r^m \dif \sigma(r)=0,\quad 1\leq m\leq M.
    \end{align*}
\end{enumerate}
\end{Lemma}
\begin{proof}
Let $\cF_M$ be the space of real polynomials with degree at most $M$ equipped with the sup-norm
\begin{align*}
    \norm{f}_\infty=\sup_{z\in [\Delta,1/2]}\abs{f(r)},\quad f\in\cF_M.
\end{align*}
Consider the linear functional $\cL_M$ defined by:
\begin{align*}
    \cL_M:\ f\mapsto f(0),\quad f\in\cF_M.
\end{align*}
By the Chebyshev extremal property, we know that
\begin{align*}
    \norm{\cL_M}\define \sup_{f\in\cF_M,\, f\nequiv0}\frac{\abs{\cL_M(f)}}{\norm{f}_\infty}=\sup_{f\in\cF_M,\, f\nequiv0}\frac{\abs{f(0)}}{\sup_{z\in[\Delta,1/2]}\abs{f(z)}}=T_M\left(-\frac{1/2+\Delta}{1/2-\Delta}\right)=\cosh(M\rho_\Delta),
\end{align*}
where $T_M$ is the degree-$M$ Chebyshev inequality of the first kind. Since $\cF_M$ is a subspace of $C([\Delta,1/2])$, according to the Hahn--Banach theorem, $\cL_M$ has an extension $\cL:\, C([\Delta,1/2])\to \bbR$ with the same norm. That is, there exists a linear functional $\cL$ on $C([\Delta,1/2])$ satisfying
\begin{align*}
    \cL(f)=\cL_M(f),\quad f\in\cF_M,
\end{align*}
and
\begin{align*}
    \norm{\cL}\define \sup_{f\in C([\Delta,1/2]),\, f\nequiv0}\frac{\abs{\cL(f)}}{\norm{f}_\infty}=\cosh(M\rho_\Delta).
\end{align*}
Then, by the Riesz representation theorem, there exists a finite signed measure $\tilde{\sigma}$ on $[\Delta,1/2]$ such that
\begin{align*}
    \cL(f)=\int_{[\Delta,1/2]}f(z)\dif \tilde{\sigma}(z),\quad f\in C([\Delta,1/2]),
\end{align*}
and thus $\norm{\tilde{\sigma}}_\mathrm{TV}=\norm{\cL}=\cosh(M\rho_\Delta)$. Meanwhile, applying the above identity with $f(r)=r^m$ for $m=0,1,2,\cdots,M$, we obtain
\begin{align*}
    \int_{[\Delta,1/2]} z^m \dif\tilde{\sigma}(z)=\delta_{m,0}.
\end{align*}
Finally, we consider the Jordan decomposition of $\tilde{\sigma}$:
\begin{align*}
    \tilde{\sigma}=\tilde{\sigma}^+-\tilde{\sigma}^-,
\end{align*}
where $\tilde{\sigma}^\pm$ are nonnegative measures. By Carath\'{e}odory's convex hull theorem, there exist two nonnegative measures $\sigma^\pm$ on $[\Delta,1/2]$, each supported on at most $M+2$ points, such that
\begin{align*}
    \int_{[\Delta,1/2]} z^m\dif \tilde{\sigma}^\pm(z)=\int_{[\Delta,1/2]} z^m\dif \sigma^\pm(z),\quad m=0,1,2\cdots,M.
\end{align*}
Therefore, as implied by the identities for $m=1,2,\cdots,M$, the signed measure $\sigma=\sigma^+-\sigma^-$ satisfies zero moments of order $1$ to $M$. Finally, the $m=0$ identity implies that
\begin{align*}
    \norm{\sigma}_\mathrm{TV}=\int_{[\Delta,1/2]}\dif \sigma^+(z)+\int_{[\Delta,1/2]}\dif \sigma^-(z)=\int_{[\Delta,1/2]}\dif \tilde{\sigma}^+(z)+\int_{[\Delta,1/2]}\dif \tilde{\sigma}^-(z)=\norm{\tilde{\sigma}}_\mathrm{TV}.
\end{align*}
The desired conclusions then follow.
\end{proof}
Then, we show that any $\sigma$ that satisfies the conditions in Lemma~\ref{lem:matching_plain} must induce a Beta mixture density that is almost uniform.
\begin{Lemma}\label{lem:almost_uniform}
    Let $\sigma$ be the signed measure as in Lemma~\ref{lem:matching_plain}. We have
    \begin{align*}
        \norm{\phi_\sigma-1}_{L^2[0,1]}\leq \frac{1}{\sqrt{3}} 2^{-M} e^{M\rho_\Delta},
    \end{align*}
    where $\phi_\sigma(r)=\int_{[\Delta,1/2]} \phi_z(r)\dif \sigma(z)$ and
    \begin{align*}
        \phi_z(r)\define \frac{1}{1-z}r^{z/(1-z)}
    \end{align*}
    is the density of $\mathrm{Beta}(1/(1-z),1)$.
\end{Lemma}
\begin{proof}
    Combining the moment identities for $\sigma$ in Lemma~\ref{lem:matching_plain} with Lemma~\ref{lem:laguerre_beta_product}, we use Fubini's theorem to claim that
    \begin{align*}
        \int_{[0,1]}\ell_m(r)\phi_\sigma(r)\dif r=\int_{z\in [\Delta,1/2]}\int_{r\in [0,1]}\ell_m(r)\phi_z(r)\dif r\dif \sigma(z)=\int_{z\in[\Delta,1/2]}z^m \dif \sigma(z)=\delta_{m,0}
    \end{align*}
    for $m=0,1,2,\cdots,M$. For $m>M$, we have
    \begin{align*}
        \abs*{\int_{[0,1]}\ell_m(r)\phi_\sigma(r)\dif r}=\abs*{\int_{z\in[\Delta,1/2]}z^m \dif \sigma(z)}\leq 2^{-m}\norm{\sigma}_\mathrm{TV}.
    \end{align*}
    Therefore, by Parseval's inequality, we have
    \begin{align*}
        \norm{\phi_\sigma-1}_{L^2[0,1]}^2=\norm{\phi_\sigma-\ell_0}_{L^2[0,1]}^2\leq \sum_{m>M}(2^{-m}\norm{\sigma}_\mathrm{TV})^2\leq \frac{2^{-2M}}{3}\cosh^2(M\rho_\Delta).
    \end{align*}
    The desired conclusion then follows.
\end{proof}

We then move on to embed the properties in Lemma~\ref{lem:matching_plain} and Lemma~\ref{lem:almost_uniform} into Gumbel--max parameters to show Proposition~\ref{prop:matching_gumbel_pivotal}.
\begin{proof}[Proof of Proposition~\ref{prop:matching_gumbel_pivotal}]
Let $\sigma$ be as in Lemma~\ref{lem:matching_plain}. When $z\in[\Delta,1/2]$ is a rational number, there exist positive integers $a$ and $b$ such that $z=a/b$. Define
\begin{align*}
    P_z=(1-z,\underbrace{1/b,1/b,\cdots,1/b}_{\text{$a$ times}}),\quad P'_z=(\underbrace{1/b,1/b,\cdots,1/b}_{\text{$b$ times}}),
\end{align*}
and let $\Sigma_z=(\delta_{P_z}-z\delta_{P'_z})/(1-z)$. Since $z\leq 1/2$, we must have $b\geq 2$. Note that both $P_z$ and $P'_z$ sum to one by our construction. Furthermore, we have $1-z\leq 1-\Delta$ since $z\in [\Delta,1/2]$ and $1/b\leq 1/2<1-\Delta$ since $h\in(0,1/2)$. Therefore, both $P_z$ and $P'_z$ are in $\cP_\Delta^{\geq 2}$, and thus $\Sigma_z$ is a signed measure on $\cP_\Delta^{\geq 2}$ satisfying
\begin{align*}
    \Sigma_z(\cP_\Delta)=1,\quad \norm{\Sigma_z}_\mathrm{TV}=\frac{1+z}{1-z}\leq 3.
\end{align*}
Also,
\begin{align*}
    g_{\Sigma_z}=\frac{1}{1-z}(g_{P_z}-zg_{P'_z})=\phi_z.
\end{align*}
Recall from Lemma~\ref{lem:matching_plain} that $\sigma$ is supported on at most $2M+4$ points. Thus, we can write
\begin{align*}
    \sigma=\sum_{i=1}^{2M+4}\alpha_i\delta_{z_i}
\end{align*}
for some $\sum_i\abs{\alpha_i}=\norm{\sigma}_\mathrm{TV}\leq \cosh(M\rho_\Delta)$ and $z_i\in [\Delta,1/2]$. Pick a small positive threshold $\eta$ to be determined later. Define the map $\cR:[\Delta,1/2]\to [\Delta,1/2]\bigcap \bbQ$ by
\begin{align*}
    \cR(z)=\begin{dcases}
        z,\quad &z\in [\Delta,1/2]\bigcap \bbQ;\\
        \text{arbitrary rational number $a/b\in [\Delta,1/2]\bigcap\bbQ$ with $\abs{z-a/b}\leq \eta$},\quad & z\in [\Delta,1/2]\backslash \bbQ.
    \end{dcases}
\end{align*}
Consider the signed measure:
\begin{align*}
    \Sigma=\int_{[\Delta,1/2]}\Sigma_{\cR(z)}\dif \sigma(z)=\sum_{i=1}^{2M+4}\alpha_i \Sigma_{\cR(z_i)}.
\end{align*}
We have
\begin{align*}
    \Sigma(\cP_\Delta)=\int_{[\Delta,1/2]} \Sigma_{\cR(z)}(\cP_\Delta)\dif \sigma(z)=1,
\end{align*}
and
\begin{align*}
    \norm{\Sigma}_\mathrm{TV}\leq \sum_{i=1}^{2M+4}\abs{\alpha_i}\norm{\Sigma_{\cR(z_i)}}_\mathrm{TV}\leq 3\cosh(M\rho_\Delta).
\end{align*}
Furthermore, by the triangle inequality and Lemma~\ref{lem:almost_uniform},
\begin{align*}
    \norm{g_\Sigma-1}_{L^2[0,1]}\leq \norm{\phi_\sigma-1}_{L^2[0,1]}+\norm{g_\Sigma-\phi_\sigma}_{L^2[0,1]}\leq \frac{2^{-M}e^{M\rho_\Delta}}{\sqrt{3}}+\norm{g_\Sigma-\phi_\sigma}_{L^2[0,1]}.
\end{align*}
It suffices to control the last term. In fact, we have
\begin{align*}
    \norm{g_\Sigma-\phi_\sigma}_{L^2[0,1]}&=\norm*{\sum_i\alpha_i\phi_{\cR(z_i)}-\sum_i\alpha_i\phi_{z_i}}_{L^2[0,1]}\\
    &\leq \sum_{i=1}^{2M+4}\abs{\alpha_i}\cdot \norm{\phi_{\cR(z_i)}-\phi_{z_i}}_{L^2[0,1]}\\
    &\leq \sum_{i=1}^{2M+4}\abs{\alpha_i}\cdot \sqrt{3}\abs{z_i-\cR(z_i)}\\
    &\leq \sqrt{3}\cosh(M\rho_\Delta)\eta,
\end{align*}
where the second-to-last line is due to
\begin{align*}
    \norm{\phi_z-\phi_w}_{L^2[0,1]}=\abs{z-w}\cdot \sqrt{\frac{(1+zw)}{(1-z^2)(1-w^2)(1-zw)}}\leq \sqrt{3}\abs{z-w}
\end{align*}
for any $z,w\in[0,1/2]$. Therefore, by setting $\eta= 2^{-M}(1/\sqrt{3}-1/3)$, the desired conclusion is proved.
\end{proof}

\subsection{Proof of Theorem~\ref{thm:lower_bound_full}}\label{apdx:proof_lower_bound_full}
Recall the special class of NTP distributions $P^{(w,S)}$ defined in \eqref{eq:pws}, indexed by the set $\cI$ specified in \eqref{eq:index_set}.
For any $P\in\cP_\Delta^{\cW}$, we use $R_P$ (resp. $S_P$) to denote the distribution of $(U,w)\in (0,1)^{\abs{\cW}}\times \cW$ given the NTP distribution $P$ without (resp. with) the Gumbel--max watermark. For any signed measure $\Xi$ on $\cI$, we use notations
\begin{align*}
    R_\Xi=\sum_{(w,S)\in\cI}\Xi(w,S) R_{P^{(w,S)}},\quad S_\Xi=\sum_{(w,S)\in\cI}\Xi(w,S) S_{P^{(w,S)}},\quad D_\Xi=R_\Xi-S_\Xi
\end{align*}
to denote the corresponding mixture measures, and use $r_\Xi$, $s_\Xi$, and $d_\Xi$ to represent their densities, respectively. Now, we specify for each $w\in\cW$ that
\begin{align*}
    \Xi_w=\sum_{\emptyset\neq S\subseteq\cW\backslash\{w\}} (-1)^{\abs{S}+1}\delta_{(w,S)}.
\end{align*}
We have $\Xi_w(\cI)=1$ due to the identity \(\sum_{j=1}^n (-1)^{j+1}\binom{n}{j}=1\) that holds for all $n\in\bbZ_+$. We further define
\begin{align*}
    \Xi=\frac{1}{\abs{\cW}}\sum_{w\in\cW}\Xi_w.
\end{align*}
Then, we also have $\Xi(\cI)=1$. Let $\cX=(0,1)^{\abs{\cW}}\times \cW$ be the space of full observation. We have the following claim:
\begin{Lemma}
    Let $\Xi$ be as above. We have
    \begin{align*}
        \norm{d_\Xi}_{L^2(\cX)}^2\leq 2(\abs{\cW}-1)!\beta_\Delta^{\abs{\cW}-1}
    \end{align*}
\end{Lemma}
\begin{proof}
We use Cauchy--Schwarz to see that
\begin{align*}
    \norm{d_\Xi}_{L^2(\cX)}^2=\norm*{\frac{1}{\abs{\cW}}\sum_w d_{\Xi_w}}_{L^2(\cX)}^2\leq \frac{1}{\abs{\cW}}\sum_{w\in\cW}\norm{d_{\Xi_w}}_{L^2(\cX)}^2.
\end{align*}
By symmetry, it suffices to show
\begin{align*}
    \norm{d_{\Xi_w}}_{L^2(\cX)}^2\leq 2(\abs{\cW}-1)!\beta_\Delta^{\abs{\cW}-1}
\end{align*}
for an arbitrary $w\in\cW$. Note that
\begin{align*}
    \norm{d_{\Xi_w}}_{L^2(\cX)}^2=\int_{(0,1)^{\abs{\cW}}} d_{\Xi_w}^2(u,w)\dif u+\sum_{v\neq w}\int_{(0,1)^{\abs{\cW}}} d_{\Xi_w}^2(u,v)\dif u.
\end{align*}
We control these two types of terms separately.
\paragraph{Case I.}We first control the norm of $d_{\Xi_w}(u,w)$. We have
\begin{align*}
    s_{\Xi_w}(u,w)&=\sum_{\emptyset\neq S\subseteq \cW\backslash\{w\}}(-1)^{\abs{S}+1}s_{P^{(w,S)}}(u,w)\\
    &=\sum_{\emptyset\neq S\subseteq \cW\backslash\{w\}}(-1)^{\abs{S}+1}\indi\left\{\frac{\log u_w}{p^{(w,S)}_w}\geq \frac{\log u_v}{p^{(w,S)}_v},\ \forall\, v\in\cW\backslash\{w\}  \right\}\\
    &=\sum_{\emptyset\neq S\subseteq \cW\backslash\{w\}}(-1)^{\abs{S}+1}\indi\left\{u_v\leq u_w^{\beta_\Delta},\ \forall\, v\in S\right\}\\
    &=\sum_{\emptyset\neq S\subseteq \cW\backslash\{w\}}(-1)^{\abs{S}+1}\prod_{v\in S}\indi\left\{u_v\leq u_w^{\beta_\Delta}\right\}\\
    &=1-\prod_{v\in\cW\backslash\{w\}}\left(1-\indi\left\{u_v\leq u_w^{\beta_\Delta}\right\}\right)\\
    &=1-\indi\left\{u_v> u_w^{\beta_\Delta},\ \forall\, v\in \cW\backslash\{w\}\right\}\\
    &\define 1-A_w(u).
\end{align*}
Let
\begin{align*}
    \pi_w=\int_{(0,1)^{\abs{\cW}}}A_w(u)\dif u.
\end{align*}
We have
\begin{align*}
    r_{\Xi_w}(u,w)&=\sum_{\emptyset\neq S\subseteq \cW\backslash\{w\}}(-1)^{\abs{S}+1}r_{P^{(w,S)}}(u,w)\\
    &=\sum_{\emptyset\neq S\subseteq \cW\backslash\{w\}}(-1)^{\abs{S}+1}p^{(w,S)}_w\\
    &=\sum_{\emptyset\neq S\subseteq \cW\backslash\{w\}}(-1)^{\abs{S}+1}\int_{(0,1)^{\abs{\cW}}}s_{P^{(w,S)}}(u,w)\dif u\\
    &=\int_{(0,1)^{\abs{\cW}}}s_{\Xi_w}(u,w)\dif u\\
    &=1-\pi_w.
\end{align*}
Therefore, we have
\begin{align*}
    \int_{(0,1)^{\abs{\cW}}}d_{\Xi_w}^2(u,w)\dif u=\int_{(0,1)^{\abs{\cW}}}(\pi_w-A_w(u))^2\dif u=\pi_w(1-\pi_w)\leq \pi_w.
\end{align*}
Meanwhile, we have
\begin{align*}
    \pi_w=\int_{(0,1)}(1-x^{\beta_\Delta})^{\abs{\cW}-1}\dif x=\frac{(\abs{\cW}-1)!\beta_\Delta^{\abs{\cW}-1}}{\prod_{j=1}^{\abs{\cW}-1}(1+j\beta_\Delta)}\leq (\abs{\cW}-1)!\beta_\Delta^{\abs{\cW}-1}.
\end{align*}

\paragraph{Case II.}We then calculate the norm of $d_{\Xi_w}(u,v)$ for every $v\neq w$. Since only subsets $S$ that contains $v$ contribute to the sum, we have
\begin{align*}
    s_{\Xi_w}(u,v)&=\sum_{\emptyset\neq S\subseteq \cW\backslash\{w\}}(-1)^{\abs{S}+1}s_{P^{(w,S)}}(u,v)\\
    &=\sum_{v\in S\subseteq \cW\backslash\{w\}}(-1)^{\abs{S}+1}s_{P^{(w,S)}}(u,v)\\
     &=\sum_{T\subseteq \cW\backslash\{w,v\}}(-1)^{\abs{T}}s_{P^{(w,T\cup\{v\})}}(u,v)\\
     &=\sum_{T\subseteq \cW\backslash\{w,v\}}(-1)^{\abs{T}}\indi\left\{\frac{\log u_v}{p^{(w,T\cup\{v\})}_v}\geq \frac{\log u_x}{p^{(w,T\cup\{v\})}_x},\ \forall\,x\in\cW\backslash\{v\}\right\}\\
     &=\indi\left\{u_w\leq u_v^{1/\beta_\Delta}\right\}\prod_{x\in \cW\backslash\{w,v\}}\left(1-\indi\left\{u_x\leq u_v\right\}\right)\\
     &=\indi\left\{u_w\leq u_v^{1/\beta_\Delta},\ u_x>u_v,\ \forall\, x\in\cW\backslash\{w,v\}\right\}\\
     &\define B_{w,v}(u).
\end{align*}
Let
\begin{align*}
    \pi_{w,v}=\int_{(0,1)^{\abs{\cW}}}B_{w,v}(u)\dif u.
\end{align*}
We similarly have
\begin{align*}
    r_{\Xi_w}(u,w)&=\int_{(0,1)^{\abs{\cW}}}s_{\Xi_w}(u,w)\dif u=\pi_{w,v}.
\end{align*}
Therefore, we have
\begin{align*}
    \int_{(0,1)^{\abs{\cW}}}d_{\Xi_w}^2(u,v)\dif u=\int_{(0,1)^{\abs{\cW}}}(\pi_{w,v}-B_{w,v}(u))^2\dif u=\pi_{w,v}(1-\pi_{w,v})\leq \pi_{w,v}.
\end{align*}
Meanwhile, we know that
\begin{align*}
    \pi_{w,v}=\int_{(0,1)}x^{1/\beta_\Delta}(1-x)^{\abs{\cW}-1}\dif x=\frac{(\abs{\cW}-2)!\beta_\Delta^{\abs{\cW}-1}}{\prod_{j=1}^{\abs{\cW}-1}(1+j\beta_\Delta)}\leq (\abs{\cW}-2)!\beta_\Delta^{\abs{\cW}-1}.
\end{align*}
Summarizing the two cases, we conclude that
\begin{align*}
    \norm{d_{\Xi_w}}_{L^2(\cX)}^2\leq \pi_w+\sum_{v\in\cW\backslash\{w\}}\pi_{w,v}\leq 2(\abs{\cW}-1)!\beta_\Delta^{\abs{\cW}-1}.
\end{align*}
\end{proof}
We can now prove Theorem~\ref{thm:lower_bound_full}.
\begin{proof}[Proof of Theorem~\ref{thm:lower_bound_full}]
Now, let $\Lambda_0$ be the uniform distribution on $\cI$. That is, $\Lambda_0(w,S)\equiv 1/(\abs{\cW}(2^{\abs{\cW}-1}-1))$ for all $(w,S)\in\cI$. Since both $\Lambda_0$ and $\Xi$ are permutation invariant. We have
\begin{align*}
    \int_\cI P^{(w,S)}\dif \Lambda_0(w,S)=\int_\cI P^{(w,S)}\dif \Xi(w,S)=(1/\abs{W},1/\abs{\cW},\cdots,1/\abs{\cW}).
\end{align*}
As a result,
\begin{align*}
    R_\Xi=\sum_{(w,S)\in\cI}\Xi(w,S) \mathrm{Unif}(0,1)^{\abs{\cW}}\times P^{(w,S)}=\mathrm{Unif}(0,1)^{\abs{\cW}}\times\mathrm{Unif}(\cW),
\end{align*}
and thus $r_\Xi(u,w)\equiv 1/\abs{\cW}$. Similarly, we have $R_{\Lambda_0}=R_\Xi=\mathrm{Unif}^{\abs{\cW}}\times \mathrm{Unif}(\cW)$ and $r_{\Lambda_0}(u,w)\equiv 1/\abs{\cW}$.

Set $\Lambda_1=(1-2\tau)\Lambda_0+2\tau\Xi$. Then we have $\Lambda_1(\cI)=1$. To make $\Lambda_1$ a probability measure on $\cI$, we need, in addition, its positivity. That is,
\begin{align*}
    (1-2\tau)\frac{1}{\abs{\cW}(2^{\abs{\cW}-1}-1)}-2\tau \frac{1}{1+\beta_\Delta}\geq 0,
\end{align*}
which can be satisfied when
\begin{align*}
    \tau\leq \frac{1}{3\abs{\cW}(2^{\abs{\cW}-1}-1)}.
\end{align*}
Define
\begin{align*}
    Q_0=(\frac{1}{2}+\tau)R_{\Lambda_0}+(\frac{1}{2}-\tau)S_{\Lambda_0},\quad Q_1=\frac{1}{2}R_{\Lambda_1}+\frac{1}{2}S_{\Lambda_1}.
\end{align*}
Then, using $R_\Lambda=R_\Xi$, we obtain
\begin{align*}
    q_1-q_0&=r_{[(1-2\tau)\Lambda_0+2\tau \Xi]/2-(1/2+\tau)\Lambda_0}-s_{-[(1-2\tau)\Lambda_0+2\tau \Xi]/2+(1/2-\tau)\Lambda_0}\\
    &=r_{\tau \Xi-2\tau \Lambda_0}-s_{-\tau \Xi}\\
    &=\tau r_{\Xi}-2\tau r_{\Lambda_0}+\tau s_{\Xi}\\
    &=-\tau r_\Xi+\tau s_\Xi.
\end{align*}
Meanwhile, we have $q_0(u,w)\geq (1/2) r_{\Lambda_0}(u,w)\geq 1/(2\abs{\cW})$. Therefore, we have
\begin{align*}
    \chi^2(Q_1,Q_0)=\int\frac{(q_1-q_0)^2}{q_0}\leq 2\abs{\cW}\tau^2\norm{d_\Xi}^2_{L^2(\cX)}\leq 4\abs{\cW}!\tau^2\beta_\Delta^{\abs{\cW}-1}.
\end{align*}
As a result, $Q_0^{\otimes N}$ and $Q_1^{\otimes N}$ are indistinguishable when
\begin{align*}
    2\abs{\cW}!\tau^2\beta_\Delta^{\abs{\cW}-1}\leq \frac{c_\delta}{N},\quad  \tau\leq \frac{1}{3\abs{\cW}(2^{\abs{\cW}-1}-1)},
\end{align*}
which yields a lower bound of:
\begin{align*}
    \Omega_{\delta,\abs{\cW}}\left(1\wedge \sqrt{\frac{1}{N\Delta^{\abs{\cW}-1}}}\right).
\end{align*}
When $\Delta\in (1/2,1-\abs{\cW})$, we use the trivial parametric lower bound $\Omega_\delta(1/\sqrt{N})$, which can be absorbed into the previous rate.
\end{proof}

\section{Auxiliary Results}
The following three properties of Laguerre polynomials can be found in \cite{abramowitz1966handbook}.
\begin{Lemma}\label{lem:laguerre_orthogonal}
Let $\{L_m\}_{m=0}^{\infty}$ be Laguerre polynomials. We have
\begin{align*}
    \int_{(0,+\infty)} L_m(x)L_n(x)e^{-x}\dif x=\indi\{m=n\}.
\end{align*}
\end{Lemma}
\begin{Lemma}\label{lem:laguerre_generating_function}
Let $\{L_m\}_{m=0}^{\infty}$ be Laguerre polynomials. We have for every $x\in(0,+\infty)$ and $t\in(-1,1)$ that
\begin{align*}
    \sum_{m=0}^\infty L_m(x) t^m=\frac{1}{1-t}\exp\left(-\frac{tx}{1-t}\right)
\end{align*}
\end{Lemma}

\begin{Lemma}\label{lem:laguerre_laplace}
Let $L_m$ be the degree-$m$ Laguerre polynomial. Then, we have
\begin{align*}
    \int_{(0,+\infty)} e^{-zx} L_m(x)\dif x=\frac{(z-1)^m}{z^{m+1}},\quad \forall\, z:\Re(z)>0.
\end{align*}
    
\end{Lemma}

\begin{Lemma}\label{lem:laguerre_beta_product}
    Let $L_m$ be the degree-$m$ Laguerre polynomial and $\ell_m(r)=L_m(-\log(r))$. For every $z\in [0,1)$ and $\phi_z(r)=r^{z/(1-z)}/(1-z)$, it holds that
    \begin{align*}
        \int_{[0,1]}\ell_m(r)\phi_z(r)\dif r=z^m.
    \end{align*}
\end{Lemma}
\begin{proof}
By change-of-variables $r=e^{-x}$ in the integral, we obtain
\begin{align*}
     \int_{[0,1]}\ell_m(r)\phi_z(r)\dif r&=\int_{(0,+\infty)}L_m(x)\cdot \frac{1}{1-z}e^{-zx/(1-z)}\cdot e^{-x}\dif x\\
     &=\int_{(0,+\infty)}L_m(x)\sum_{n=0}^\infty L_n(x)z^n\cdot e^{-x}\dif x\\
     &=\sum_{n=0}^\infty z^n\int_{(0,+\infty)}L_m(x)L_n(x)e^{-x}\dif x\\
     &=z^m,
\end{align*}
where we have used Lemma~\ref{lem:laguerre_generating_function} and Fubini's theorem.
\end{proof}

\begin{Lemma}[Cauchy's estimate, see \cite{stein2010complex} Chapter 2 Corollary 4.3]\label{lem:cauchy_estimate}
Let $f$ be a holomorphic function in the open disk $\bbD(0,r)=\{z\in\bbC:\, \abs{z}< r\}$. Suppose that there exists a $A\in[0,+\infty)$ such that $\abs{f(z)}\leq A$ for all $z\in\bbD(0,r)$. Then, for each $m\in\bbN$,
\begin{align*}
    \abs{f^{(m)}(0)}\leq \frac{m!}{r^m}A.
\end{align*}
\end{Lemma}

\begin{Lemma}[Hadamard three-lines theorem, see \cite{stein2010complex} Chapter 4 Problem 3]\label{lem:hadamard}
Let $a<b$ be any two real numbers. Let $f:\bbC\to\bbC$ be a bounded function on the strip $\{z\in\bbC:a\leq \Re z\leq b\}$, holomorphic in the interior and continuous on the whole strip. Define the function $G(x)=\sup_{z:\Re z=x}\abs{f(z)}$ for $x\in[a,b]$. Then, for any $x=ta+(1-t)b$ with $t\in[0,1]$, we have
\begin{align*}
    G(x)\leq G(a)^t G(b)^{1-t}.
\end{align*}
    
\end{Lemma}

\end{sloppypar}

\end{document}